\documentclass[10pt,a4paper]{article}
\usepackage{amsmath,amsfonts,amssymb,amsthm,graphicx,dsfont}
\usepackage[latin1]{inputenc}
\usepackage{booktabs,color}
\usepackage{algorithm,algorithmic}
\usepackage[left=3cm,right=3cm,top=1.5cm,bottom=2cm]{geometry}
\usepackage[normalem]{ulem}

 \theoremstyle{plain}
 \newtheorem{theorem}{Theorem}
 \newtheorem{lemma}[theorem]{Lemma}
 \newtheorem{proposition}[theorem]{Proposition}

\newtheorem{remark}[theorem]{Remark}

\DeclareMathAlphabet{\mathbf}{OT1}{cmr}{bx}{it}

\newcommand{\mc}{\mathcal}

\newcommand{\bs}[1]{\boldsymbol #1}

\DeclareMathAlphabet{\mathbf}{OT1}{cmr}{bx}{it}

\newcommand{\argmax}{\operatornamewithlimits{arg\ max}}
\newcommand{\essinf}{\operatornamewithlimits{ess\ inf}}

\newcommand{\ve}{\mathbf e}

\newcommand{\vi}{\mathbf i}
\newcommand{\vj}{\mathbf j}
\newcommand{\vk}{\mathbf k}

\newcommand{\vp}{\mathbf p}

\newcommand{\vx}{\mathbf x}
\newcommand{\vy}{\mathbf y}
\newcommand{\vz}{\mathbf z}

\newcommand{\vrho}{\boldsymbol{\rho}}

\newcommand{\vGamma}{\boldsymbol{\Gamma}}

\newcommand{\bbN}{\mathbb{N}}
\newcommand{\bbR}{\mathbb{R}}

\renewcommand{\d}{\mathrm{d}}

\newcommand{\F}{\mathcal F}
\newcommand{\Cand}{\mathcal C}
\newcommand{\Mark}{\mathcal M}
\newcommand{\Marg}[1]{\mathrm{Marg}(#1)}
\newcommand{\RMarg}[1]{\mathrm{R}(#1)}

\newcommand{\m}{\mathfrak m}

 \newcommand{\rev}[1]{\textcolor{black}{#1}}
 \newcommand{\rewrite}[2]{\textcolor{black}{#2}}

\begin{document}
\title{On the convergence of adaptive stochastic collocation for elliptic partial differential equations with affine diffusion}
\author{Martin Eigel$^1$, Oliver Ernst$^2$, Bj\"orn Sprungk$^3$, Lorenzo Tamellini$^4$\\
{\footnotesize  $^1$Weierstrass Institute, Berlin, Germany}\\
{\footnotesize  $^2$Department of Mathematics, TU Chemnitz, Germany}\\ 
{\footnotesize  $^3$Faculty of Mathematics and Computer Science, TU Bergakademie Freiberg, Germany}\\ 
{\footnotesize  $^4$Istituto di Matematica Applicata e Tecnologie Informatiche ``E. Magenes'', Pavia, Consiglio Nazionale delle Ricerche, Italy} 
}

\maketitle

\begin{abstract}
Convergence of an adaptive collocation method for the parametric stationary diffusion equation with finite-dimensional affine coefficient is shown.
The adaptive algorithm relies on a recently introduced residual-based reliable a posteriori error estimator.
For the convergence proof, a strategy recently used for a stochastic Galerkin method with a hierarchical error estimator is transferred to the collocation setting.
Extensions to other variants of adaptive collocation methods (including the now classical approach proposed in ``T. Gerstner and M. Griebel, Dimension-adaptive tensor-product quadratuture, Computing, 2003'') \rev{are} explored.
\end{abstract}

\paragraph{Keywords.} Random PDEs, parametric PDEs, sparse grids, stochastic collocation, adaptive algorithms, high-dimensional approximation, high-dimensional interpolation

\paragraph{AMS subject classifications:} 
	65D05, 
    65D15, 
    65C30, 
    60H25 
 

\section{Introduction} \label{sec:introduction}

Collocation methods are now a mainstay for solving equations containing high-dimensional parameters such as arise in uncertainty quantification (UQ) analyses of ordinary or partial differential equations (ODE/PDE) with uncertain model coefficients
\cite{MathelinHussani2003TR,XiuHesthaven2005,BabuskaEtAl2007}.
It was realized early on that already moderately high-dimensional problems become tractable only
when the approximations are based on sparse subspaces of the basic tensor product construction 
\cite{NobileEtAL2008a,NobileEtAl2008b,BieriSchwab2009,MaZabaras2009,Bieri2011,BeckEtAl2012}.

Subsequent work established that, under mild conditions, certain classes of random PDEs are tractable even in presence of countably many parameter variables
\cite{cohen.devore.schwab:nterm1,cohen.devore.schwab:nterm2,SchillingsSchwab2013,ChkifaEtAl2014,ChkifaEtAl2014,BachmayrEtAl2017,zech.schwab:convergence,
  HoangSchwab2014,bachmayr:partIIlognormal,chen:adaptive.lognormal,ErnstEtAl2018}.
These results prove that \emph{there exists} a sequence of converging approximation operators (be they of collocation or Galerkin/projection nature) and derive the corresponding convergence rates.
Such sequences of converging approximation operators can be sometimes estimated a priori as in
\cite{zech.schwab:convergence,chen:adaptive.lognormal,ErnstEtAl2018}.
Another possible procedure is to rely instead on a posteriori adaptive strategies:
the details of these strategies vary depending on the type of approximation operators (projection/collocation) and, moreover, these a posteriori adaptive strategies are often based on heuristics known to behave well in practice (even better than the a priori constructions) but for which a proof of convergence is \rev{often} lacking.

For projection approaches, adaptive stochastic Galerkin finite element methods (ASGFEM), which control the discretization of both physical and parametric variables, are well-studied.
The extensive research activity in the last years comprises in particular residual-based error estimators
\cite{EigelEtAl2014,EigelEtAl2015,eigel2016local,eigel2017adaptive} and hierarchical error estimators
\cite{BespalovEtAl2014,BespalovSilvester2016,CrowderEtAl2019,BespalovEtAl2019}.
The setting in these works is similar to the one considered here, i.e., linear elliptic PDEs with affine parametric coefficients.
However, the cited works allow for a countably infinite expansion, which makes an additional dimension adaptivity necessary.
With the employed Legendre chaos discretization for the parameter space, only the margin of an active set of polynomials has to be considered in the error estimator.
The developed error estimators have been shown to be reliable and efficient, which for hierarchical estimators usually requires additional assumptions.
Convergence of an ASGFEM algorithm was first shown in~\cite{EigelEtAl2015} for a residual estimator and, using a different argument, in~\cite{BespalovEtAl2019} for a hierarchical estimator.
A goal-oriented error estimator was presented in~\cite{BespalovEtAl2019a} and the more involved case of nonlinear coefficients and Gaussian parameters has only been considered recently in~\cite{eigel2020adaptive} with a low-rank hierarchical tensor discretization.

On the stochastic collocation side, the current literature discusses quite extensively algorithms for stochastic adaptivity, whereas much less attention has been devoted to (reliable) spatial adaptivity.
To date, most adaptive sparse grid approximation schemes involve some variation of the basic procedure proposed by Gerstner and Griebel in \cite{GerstnerGriebel2003}, see also \cite{Hegland:adaptive}.
This algorithm drives adaptivity in the parameter variables by exploring at each iteration a certain number of \rewrite{admissible}{} sparse subspaces \rev{admissible} to the approximation and then evaluating for each of these an \emph{error indicator}; this requires solving a certain number of PDEs. 
The subspace with the largest error indicator is selected and added to the approximation, and a new set of admissible sparse subspaces for the next enrichment step is generated. 
Several error indicators and variations of the selection strategy have been considered,
see e.g. \cite{klimke:thesis,GriebelKnapek2009,SchillingsSchwab2013,ChkifaEtAl2014,NobileEtAl2016,farcas:adaptivity}.
A crucial point is that these error indicators are \emph{heuristics}.
Conversely, the work \rewrite{by}{} \cite{GuignardNobile2018} \rev{by Guignard and Nobile}
proposes a variation of the Gerstner--Griebel algorithm based on a reliable residual-based error \emph{estimator} which can control adaptivity in both the physical and parametric variables.
Another significant difference compared with typical indicator-based adaptive algorithms
is that the procedure proposed in \cite{GuignardNobile2018} evaluates the error estimator \emph{without solving additional PDEs}.
This allows significant computational savings with relative to the basic Gerstner--Griebel algorithm.
For other works discussing spatial adaptivity in the context of stochastic collocation methods, see \cite{SchiecheLang2014,LangEtAl2019TR}.

\rev{Guignard and Nobile} give no convergence analysis in \cite{GuignardNobile2018} for \rewrite{the}{their} proposed algorithm, and our contribution in this work is to close this gap.
We do this by proving convergence of a slight modification of \rewrite{the algorithm in [33]}{their algorithm (cf. Algorithm \ref{alg:asc}), }
thus establishing \rewrite{the first (to our knowledge)}{a} convergence result for an adaptive sparse collocation method. 
\rev{This result is stated in Theorem \ref{theo:GN_conv}.}
Our convergence analysis is based on a convergence theorem for abstract adaptive approximations (i.e., which covers both projection and collocation approximations, as well as other possible approximation strategies) w.r.t. the parameter variables. 
We derive this theorem by generalizing results given in \cite{BespalovEtAl2019}
on convergence of adaptive stochastic Galerkin methods.
This approach for proving convergence requires that the employed error estimator possesses the property of \emph{reliability}.
In \cite{GuignardNobile2018} \rewrite{the authors}{Guignard and Nobile} already established this property for their error estimator, but only for a specific model problem, namely, an elliptic PDE whose diffusion coefficient depends linearly on a finite number of parameters.
Moreover, we also require the underlying univariate sequence of collocation points to be nested in order that the sparse collocation construction be interpolatory.
Hence, our particular convergence result is also tied to these assumptions on the underlying PDE and collocation points.
However, we believe that the general approach for establishing convergence of adaptive sparse collocation methods presented in this paper might be adapted to more general cases in the future. For instance, upon assuming that the error indicator used in the \rev{basic} Gerstner--Griebel  adaptive algorithm is indeed a reliable error estimator, we are able to prove convergence of this variant of the algorithm as well \rev{(see Theorem \ref{theo:GG_alg})}.
We note that our analysis considers adaptivity in the parameter variables only, i.e., we focus on the semi-discrete setting.
Finally, we mention the simultaneous and independent work \cite{Feischl2020},
which also provides a convergence result (and a convergence rate)
for adaptive stochastic collocation methods applied to an elliptic PDE with diffusion coefficient depending affinely on finitely many random variables. While the overall framework and the focus of that work is similar to ours, some differences are noteworthy: the algorithm for which \cite{Feischl2020} proves convergence is
essentially \rewrite{that}{the one} discussed \rev{by Guignard and Nobile} in \cite{GuignardNobile2018} while we consider a different version and, in addition, we
also provide a convergence proof for the original Gerstner--Griebel variant. 
Furthermore, the line of proof in \cite{Feischl2020}, while similar to the present one, has of course some different technical aspects:
\rev{in particular, our proof is valid for any choice of collocation points over the parameter space, whereas the proof in \cite{Feischl2020}
 assumes that Clenshaw--Curtis collocation points are used when constructing the sparse grid.}

The remainder of this paper is structured as follows. 
Sections \ref{sec:setting} and \ref{sec:collocation} contain preliminary information: in particular, Section~\ref{sec:setting} states the model problem and recalls the results in \cite{BespalovEtAl2019} that will be instrumental for the rest of the work, while
Section~\ref{sec:collocation} gives details on the construction of adaptive sparse grid collocation schemes.
Sections~\ref{sec:convergence} and \ref{sec:proof} contain our main results:
Section~\ref{sec:convergence} contains the statement of the specific adaptive collocation algorithm
\rev{that we consider (i.e., our version of the Guignard--Nobile algorithm, see Algorithm \ref{alg:asc})},
the associated convergence result \rev{(Theorem \ref{theo:GN_conv}), the convergence result of the Gerstner--Griebel Algorithm (Theorem \ref{theo:GG_alg})},
and some discussion on computational aspects, while Section~\ref{sec:proof} contains the proof of the convergence result.
Finally, conclusions and future research directions are outlined in Section~\ref{sec:conclusions}.

\section{Preliminaries} \label{sec:setting}

In this section we specify the model problem under consideration and recall basic properties of its solution.
Furthermore, we discuss general adaptive approximations w.r.t.~the 
parameter variables and state an abstract convergence result which provides the basis of our convergence analysis for adaptive sparse grid collocation.

\subsection{Model Problem}\label{sec:model}

We consider a common model problem arising in uncertainty propagation via random differential equations, i.e., the stationary diffusion equation containing a coefficient function which depends linearly on a high-dimensional parameter.
Specifically, we wish to solve the parametric elliptic boundary value problem 
\begin{subequations} \label{eq:bvp}
\begin{align}
	- \nabla \cdot \left(a(\vy) \nabla u(\vy) \right) &= f, && \text{ on } D\subset \bbR^d\\
	\qquad
	u(\vy) &= 0, && \text{ on } \partial D.
\end{align}
\end{subequations}
The  domain $D \subset \bbR^d$ is assumed to be bounded and Lipschitz, $f\in L^2(D)$ and the coefficient $a(\vy) \in L^\infty(D)$ is given by
\begin{equation}
\label{equ:affine-a}
	a(\vx, \vy)
	= 
	a_0(\vx) + \sum_{m=1}^M a_m(\vx) \ y_m,
	\qquad
	\vy \in \vGamma := \Gamma^M, \; \Gamma := [-1,1],
\end{equation}
where $M\in\bbN$ is a finite number and $a_0, \ldots, a_M \in L^\infty(D)$.
The parametric domain $\vGamma$ is equipped with a uniform product measure $\mu(\d \vy) := \bigotimes_{m=1}^M \frac{\d y_m}{2}$, i.e., the components of $\vy$ can be viewed as i.i.d.~uniform random variables over $\Gamma =[-1,1]$.
Further, we assume that the functions $a_0, \ldots, a_M \in L^\infty(D)$ satisfy the \emph{uniform ellipticity condition}
\begin{equation}\label{equ:UEA}
	\sum_{m=1}^M |a_m(\vx)| \leq a_0(\vx) - r, \qquad \forall x\in D,
\end{equation}
for some $r>0$.
This implies that
\begin{equation}
\label{eq:UEA}
	a_{\min}
	:=
	\min_{\vy \in \vGamma} \;
	\essinf_{\vx \in D} \;
	a(\vx, \vy)
	\geq r > 0.
\end{equation}
We then define the constant
\begin{equation}\label{equ:alpha}
\alpha := 1 - \frac {a_{\min}}{\inf_{x\in D} a_0(\vx)} \in (0,1),
\end{equation}
which will turn out to be important in Theorem~\ref{cor:Bachmayr_finite} below.
Due to the uniform ellipticity assumption, the weak solution $u(\vy)\in \mathcal{H}=H_0^1(D)$
exists for any $\vy \in \vGamma$ and satisfies $u \in C(\vGamma; \mathcal{H})$. 

\paragraph{Polynomial expansions}
In order to approximate the solution $u$ of \eqref{eq:bvp}, or rather the parameter-to-solution map $\vy \mapsto u(\cdot,\vy) \in \mathcal{H}$,
we shall analyze polynomial expansions of $u$ in the parameter $\vy\in \vGamma$, 
\begin{equation}\label{equ:Poly_Exp}
	u(\vx, \vy) 
	=
	\sum_{\vk \in \F} u_\vk(\vx)  P_\vk(\vy), 
	\qquad
	\F := \bbN_0^M, \quad u_\vk \in \mathcal{H},
\end{equation}
where $P_\vk(\vy) = \prod_{m=1}^M P_{k_m}(y_m)$ is a finite product of univariate polynomials $P_{k}\colon \Gamma \to \bbR$ of degree $k$ with $P_0 \equiv 1$.
Two common choices for the basic polynomials $P_k$ are 
\begin{enumerate}
\item
\emph{Taylor polynomials:} $P_\vk(\vy) := \vy^\vk = \prod_{m=1}^M y_m^{k_m}$ where then
\[
	u_\vk(\vx) = t_\vk(\vx) := \frac{1}{\vk!} \partial^{\vk}u(\vx, \boldsymbol 0),
\]

\item
\emph{Legendre polynomials:} $P_\vk(\vy) := L_\vk(\vy) = \prod_{m=1}^M L_{k_m}(y_m)$ with $L_k$ denoting the $k$th $L^2_{\mu_1}$-normalized Legendre polynomial w.r.t.~the uniform distribution $\mu_1(\d x) = \frac{\d y}{2}$ on $\Gamma =[-1,1]$ and
\[
	u_\vk(\vx) := \int_{\vGamma} u(\vx,\vy) L_\vk(\vy)\ \mu(\d \vy).
\]
\end{enumerate}

Since $u \in C(\vGamma; \mathcal{H}) \subset L^2_\mu(\vGamma;\mathcal{H})$
we have that the expansion \eqref{equ:Poly_Exp} using Legendre polynomials converges in $L^2_\mu(\vGamma;\mathcal{H})$.
The following result due to \cite{BachmayrEtAl2017} establishes under suitable assumptions an $\ell^p$-summability
of both Taylor and Legendre coefficients which, for instance,
implies that the Taylor expansion \eqref{equ:Poly_Exp} of $u$ converges in $L^\infty(\vGamma;\mathcal{H})$.


\begin{theorem}[{\cite[Theorem 2.2 \& 3.1, Corollary 2.3 \& 3.2]{BachmayrEtAl2017}}]\label{cor:Bachmayr_finite}
Let the condition \eqref{equ:UEA} for $a$ as in \eqref{equ:affine-a} be satisfied.
Then a unique solution $u$ of the corresponding elliptic problem \eqref{eq:bvp} exists and belongs to $C(\vGamma; \mathcal{H})$.
Moreover, for any $\vrho := (\rho_m)_{m=1}^M$ with $1 < \rho_m < \alpha^{-1}$ with $\alpha$ as in \eqref{equ:alpha} 
\begin{enumerate}
\item
the Taylor coefficients $t_\vk \in \mathcal{H}$ of $u$ satisfy $(\vrho^{\vk}\|t_\vk\|_{\mathcal{H}})_{\vk\in\F} \in \ell^2(\F)$,
		
\item
  and the Legendre coefficients $u_\vk \in \mathcal{H}$ of $u$ satisfy
  $(b^{-1}_\vk \vrho^{\vk} \|u_\vk\|_{\mathcal{H}})_{\vk\in\F} \in \ell^2(\F)$ with $b_\vk := \prod_{m=1}^M\sqrt{1+2k_m}$.
	\end{enumerate}
\end{theorem}

\begin{remark}
  The authors of \cite{BachmayrEtAl2017} actually consider the infinite-dimensional noise case, i.e.,
  with $M=\infty$ in \eqref{equ:affine-a}, and prove the results stated in Theorem \ref{cor:Bachmayr_finite} under the assumption that
\[
	\left\| \frac{\sum_{m=1}^\infty \rho_m |a_m|}{a_0}\right\|_{C(D)} < 1,
\]
for a sequence $\vrho := (\rho_m)_{m\geq 1}$ with $\rho_m>1$.
Hence, Theorem \ref{cor:Bachmayr_finite} can be derived easily from this general case by setting $a_m(\vx) \equiv 0$ and $\rho_m > 1$ arbitrarily for $m>M$:
	\begin{align*}
	\left\| \frac{\sum_{m=1}^\infty \rho_m |a_m|}{a_0}\right\|_{C(D)}
	& = \left\| \frac{\sum_{m=1}^M \rho_m |a_m|}{a_0}\right\|_{C(D)} 
	< \alpha^{-1} \left\| \frac{\sum_{m=1}^\infty |a_m|}{a_0}\right\|_{C(D)}
	\leq \alpha^{-1} \ (1 - a_{\min}) = 1.
	\end{align*}
\end{remark}

\subsection{Adaptive Polynomial Approximation} \label{sec:adaptive_algorithms}

Given the decay rate  stated in Theorem \ref{cor:Bachmayr_finite} for the norms of the coefficients $u_\vk$ 
of the expansion \eqref{equ:Poly_Exp}, 
a polynomial approximation of $u$ seems feasible.
To this end, we consider the truncated expansions $u_\Lambda$ based on a finite multi-index set $\Lambda \subset \F$,
\[
	u_\Lambda := 
	S_\Lambda u
	=
	\sum_{\vk \in \Lambda} \widehat{u}_\vk P_\vk,
	\qquad
	 \widehat{u}_\vk \in \mc H,
\]
where $S_\Lambda$ denotes a suitable \emph{approximation operator} and $\widehat{u}_\vk$ are approximations to the true coefficients ${u}_\vk$ of $u$ (cf.\ \eqref{equ:Poly_Exp}).
For instance, $S_\Lambda$ could be the operator associated with a Galerkin approach for approximating $u$ using the finite-dimensional polynomial space
\[
	\mc P_\Lambda(\vGamma) := \mathrm{span}\left\{ P_\vk \colon \vk \in \Lambda\right\},
\]
or, as we in our case later, the operator associated to sparse grid collocation based on $\Lambda$.
At this point we do not need to further specify $S_\Lambda$.

We consider in particular an \emph{adaptive approach} to compute such polynomial approximations $u_\Lambda$.
More specifically, starting from an initial set $\Lambda_0 \subset \F$ we construct nested multiindex sets $\Lambda_n \subset \Lambda_{n+1}$, $n\in\bbN_0$,
and compute the associated polynomial approximations $u_n := S_{\Lambda_n} u$
by the generic adaptive algorithm detailed in Algorithm \ref{alg:aaa}.
\begin{algorithm}[t]
  \caption{Generic adaptive algorithm}\label{alg:aaa}
  \begin{algorithmic}
    \STATE $\Lambda_0 = \{\boldsymbol 0\}$
    \STATE $u_0 := S_{\Lambda_0}u$
    \FOR{ $n\in\bbN_0$} 
    \STATE Choose a \emph{candidate set} of multi-indices 
           $\Cand_n \subset \F \setminus \Lambda_n$ for enriching $\Lambda_n$
    \STATE Evaluate estimates of the error contribution on the candidate set: 
    \[
      \eta_n(\vk) = \eta(\vk, u_n), \quad \vk \in \Cand_n
    \]
    \STATE Determine \emph{marked indices} $\Mark_n \subset \Cand_n$
    (according to a given marking strategy based on $\eta_n(\vk)$);
    \STATE Set $\Lambda_{n+1} := \Lambda_n \cup \Mark_n$
    \STATE Set $u_{n+1} := S_{\Lambda_{n+1}}u$.
    \ENDFOR
  \end{algorithmic}
\end{algorithm}
Again, we do not further specify how to compute the estimates  $\eta_n(\vk) = \eta(\vk, u_n)$ at this point.
Instead, we provide a fairly general convergence theorem for Algorithm \ref{alg:aaa},
stating conditions on $\eta_n(\vk)$ that guarantee convergence of the algorithm.

The following theorem draws upon the work \cite{BespalovEtAl2019} on the convergence of adaptive stochastic Galerkin methods.
Specifically, it is a compact summary of a way of proving for convergence for stochastic Galerkin outlined in detail in \cite[Section 6 and 7]{BespalovEtAl2019},
slightly modified to fit the application to adaptive sparse grid collocation.
We state the theorem here and provide the proof at the end of the section.

\begin{theorem}[cf. \cite{BespalovEtAl2019}]\label{theo:abstract_adapt}
Let $u_n$ denote the approximations constructed via Algorithm \ref{alg:aaa}. 
Assume that
\begin{enumerate}
\item
the total error estimator $\eta_n := \sum_{\vk \in \Cand_n} \eta_n(\vk)$ is reliable, i.e., there exists a constant $C<\infty$ independent of $n$ such that
\[
	\|u - u_n\|
	\leq C \eta_n,
\]
where $\|\cdot\|$ denotes a suitable norm for functions $v\colon \vGamma \to \mc H$,

\item
there exists a sequence of non-negative numbers $(\eta_\infty(\vk))_{\vk \in \F} \in \ell^1(\F)$ such that for $( \widehat \eta_n(\vk) )_{\vk \in \F}$ with $\widehat \eta_n(\vk) := \eta_n(\vk)$ for $\vk \in \Cand_n \cup \Lambda_n$ and $\widehat \eta_n(\vk) = 0$ otherwise, we have
\[
	\lim_{n\to \infty} \|\eta_\infty  - \widehat  \eta_n\|_{\ell^1} = 0,
\]

\item
there exists a constant $c > 0$ independent of $n$ such that for all $\vk \in \Cand_n\setminus \Mark_n$ we have
\[
	\eta_{n}(\vk)
	\leq
	c \sum_{\vi \in \Mark_n} \eta_n(\vi).
\]
\end{enumerate}
From these assumptions it follows that
\[
	\lim_{n\to\infty} \|u - u_n\| = 0.
\]
\end{theorem}


\begin{remark}\label{rem:abstract_theorem}
Before we prove the theorem, we comment on the second and third assumption: 
\begin{enumerate}
\item
The third assumption is generally easily to satisfy.
For instance, simply choosing $\Mark_n := \argmax_{k\in \Cand_n} \eta_n(\vk)$ satisfies the assumption with $c=1$. 

\item
For sparse grid collocation, the second assumption turns out to be the most difficult to verify.
Moreover, it is probably the most cryptic assumption of the theorem.
It can usually be verified as follows: 
assuming the sequence $u_n$ has a limit $u_\infty$ with corresponding error estimators $\eta_\infty(\vk) := \eta(\vk, u_\infty)$,
conclude from $u_n\to u_\infty$ that $\|\eta_\infty  - \widehat  \eta_n\|_{\ell^1}\to 0$ by exploiting continuity properties
of the error estimator $\eta(\vk, u_n)$ w.r.t.~$u_n$.
Note that in principle $u_\infty$ is just the limit of $u_n$, but does not necessarily coincide with the actual solution
of the PDE (\ref{eq:bvp}). The fact that $u_\infty=u$ is the asesrtion of the theorem.

\item
As we will see in the proof of Theorem \ref{theo:abstract_adapt}, the second assumption represents some kind of saturation of the reliable error estimators $\eta_n$: since $\|\eta_\infty  - \widehat  \eta_n\|_{\ell^1}\to 0$ we have that 
\[
	\eta_n 
	\leq 
	\sum_{\vk \in \Cand_n} \eta_\infty(\vk) 
	  + \sum_{\vk \in \Cand_n} |\eta_n(\vk) - \eta_\infty(\vk)|
	\leq 
	\sum_{\vk \in \Cand_n} \eta_\infty(\vk)  
	  +  \| \widehat \eta_n - \eta_\infty\|_{\ell^1}
\]
converges to zero if $\sum_{\vk \in \Cand_n} \eta_\infty(\vk)$ does.
Since $\| \widehat \eta_n - \eta_\infty\|_{\ell^1} < \infty$ we can expect $\eta_\infty(\vk)$ to decay for large multi-indices $\vk$.
Thus, if $\Cand_n$ tends to include increasingly larger multi-indices $\vk$, then $\sum_{\vk \in \Cand_n} \eta_\infty(\vk)$ should decay to zero.
This will be made rigorous in the subsequent proof. 
\end{enumerate}
\end{remark}
 
\noindent The proof of Theorem \ref{theo:abstract_adapt} employs the following abstract lemma which was shown for the case $p=2$ in \cite[Lemma 15]{BespalovEtAl2019}.
Since their proof can be generalized to arbitrary $1\leq p < \infty$ without any significant modification we simply state the result and refer to \cite[Lemma 15]{BespalovEtAl2019} for a detailed proof.

\begin{lemma}[{cf. \cite[Lemma 15]{BespalovEtAl2019}}]\label{lem:tech}
Let $\vz = (z_k)_{k\in\bbN} \in \ell^p(\bbN)$, $p\in[1,\infty)$, and $\vz^{(n)} = (z^{(n)}_k)_{k\in\bbN} \in \ell^p(\bbN)$, $n\in\bbN_0$, be sequences of non-negative numbers satisfying $\lim_{n\to \infty} \|\vz - \vz^{(n)}\|_{\ell^p}= 0$.
Assume further that there exists a continuous function $g\colon [0,\infty) \to [0,\infty)$ with $g(0)=0$  and a sequence of nested subsets $\mc J_n \subset \bbN$, i.e., $\mc J_n \subset \mc J_{n+1}$, such that 
\[
	\forall n \in\bbN_0 \ \forall k \notin \mc J_{n+1}\colon \ 
	 z_k^{(n)} 
	 \leq 
	 g \left( \sum_{i \in \mc J_{n+1}\setminus \mc J_n} \left( z_i^{(n)} \right)^p \right).
\]
Then $\lim_{n\to\infty} \sum_{k \notin \mc J_n} z_k^p = 0$.
\end{lemma}

\begin{proof}[\textbf{Proof of Theorem \ref{theo:abstract_adapt}}]
Since the error estimator is reliable, we only need to show that
\[
	\lim_{n\to \infty} \eta_n
	=
	\lim_{n\to \infty} \sum_{\vk \in \Cand_n} \eta_n(\vk)
	=
	0. 
\]
Due to 
\begin{align*}
	\sum_{\vk \in \Cand_n} \eta_n(\vk)
	& \leq \sum_{\vk \in \Cand_n} \eta_\infty(\vk) + \sum_{\vk \in \Cand_n} |\eta_n(\vk) - \eta_\infty(\vk)|
	\leq \sum_{\vk \in \Cand_n} \eta_\infty(\vk)  +  \| \widehat \eta_n - \eta_\infty\|_{\ell^1(\bbN)},
\end{align*}
as well as $\| \widehat \eta_n - \eta_\infty\|_{\ell^1}\to 0$ by assumption, the statement of the theorem follows if
\[
	\lim_{n\to\infty} \sum_{\vk \in \Cand_n} \eta_\infty(\vk) = 0.
\]
In order to show this we apply Lemma \ref{lem:tech} as follows:
we identify the countable set $\F$ with $\bbN$, $\eta_\infty$ with $\vz$ and $\widehat \eta_n$ with $\vz^{(n)}$.
Recall that by assumption 
$\| \widehat{\eta}_n - \eta_\infty \|_{\ell^1} \rightarrow 0$.
Thus, the first assumption of Lemma \ref{lem:tech} is satisfied.
Moreover, we identify the $\Lambda_n\subset \F$ with $\mc J_n\subset \bbN$.
These sets are nested and $\mc J_{n+1}\setminus \mc J_n$ corresponds to $\Mark_n$.
By our third assumption and the construction of $\widehat \eta_n$ there holds for each $n\in\bbN$
\[
	\widehat \eta_n(\vk)
	\leq
	c \sum_{\vi \in \Mark_n} \widehat \eta_n(\vi)
	\qquad
	\forall \vk \notin \Lambda_{n+1}, 
\]
since $\widehat \eta_n(\vk) = 0$ for $\vk \notin \Cand_n \cup \Lambda_n$ and
$(\Cand_n \cup \Lambda_n)\setminus\Lambda_{n+1} = \Cand_n\setminus \Mark_n$.
Thus, the second assumption of Lemma \ref{lem:tech} is also satisfied with $g(s) = cs$.
Hence, we can apply Lemma \ref{lem:tech} to $\vz \simeq \eta_\infty$ and $\vz_n \simeq \widehat \eta_n$ and obtain that
\[
	\lim_{n\to\infty} \sum_{\vk \notin \Lambda_{n}} \eta_\infty(\vk) = 0,
\]
which by $\sum_{\vk \in \Cand_n} \eta_\infty(\vk) \leq \sum_{\vk \notin \Lambda_{n}} \eta_\infty(\vk)$ concludes the proof.
\end{proof}


\section{Adaptive Sparse Collocation}\label{sec:collocation}

We now introduce the sparse collocation approach and discuss how adaptive sparse grid algorithms can be derived from the abstract Algorithm \ref{alg:aaa}.
In particular, we show how to obtain the classical a-posterior adaptive algorithm
\rewrite{from}{by Gerstner and Griebel} \cite{GerstnerGriebel2003} based on heuristic error \emph{indicators}
(as opposed to reliable error \emph{estimators}, as \rev{proposed by Guignard and Nobile} in \cite{GuignardNobile2018}).
As already discussed in the introduction, changing from indicators to estimators is key to proving convergence.
Our version of the estimator-based algorithm \rev{by Guignard and Nobile} and its convergence are then discussed in the subsequent sections.

\paragraph{Univariate interpolation nodes}

The first ingredient for any sparse grid construction is the choice of the underlying univariate sequences of collocation points. 
In this work, we consider nested point sequences:
Let $(y_{(i)})_{i\in\bbN_0} \subset [-1,1]$ denote a sequence of univariate interpolation nodes and define the associated node sets
\begin{equation}\label{eq:node_sets}
	\mc Y_k := \{y_{(i)}\colon i = 0,\ldots, \m(k) \} \subset \Gamma,
	\qquad
	k\in\bbN_0,
\end{equation}
where $\m \colon \bbN_0 \to \bbN_0$ denotes the \emph{growth function} of the sets $\mc Y_k$, i.e., $|\mc Y_k| = 1 + \m(k)$.
We assume throughout that $\m(0) = 0$ and that $\m$ is strictly increasing.
Thus, we exclude \emph{delayed sequences} of node sets with $\mc Y_k = \mc Y_{k+1}$ for certain $k$ as sometimes employed for sparse grid methods, see \cite{Petras2003}.
As an immediate consequence of our assumption, we have $\m(k) \geq k$ and $|\mc Y_k| \geq k+1$.
We later also use the \emph{generalized inverse} of the growth function given for $i \in \bbN_0$ by
\begin{equation}\label{eq:m_inverse}
	\m^{-1}(i) 
	:= 
	\min\{ k \in \bbN_0 : i \leq \m(k)\} \leq i,
\end{equation}
which gives the index of the first node set $\mc Y_k$ which contains $y_{(i)}$.
A particularly convenient construction of such nested nodes is provided by Leja points.
Leja sequences on $\Gamma=[-1,1]$ are defined recursively by first choosing $y_{(0)} \in \Gamma$ and then setting
\begin{equation} \label{eq:leja.def}
    y_{(k)} = \argmax_{y \in \Gamma} \prod_{i=0}^{k-1} |y-y_{(i)}|,  
    \qquad k \in \bbN_0,
\end{equation}
see e.g. \cite{Chkifa2013,ChkifaEtAl2014,Chkifa2015,SchillingsSchwab2013,nobile.etal:leja-points} and the references therein.
The standard choice is to set $y_{(0)}=-1$; the rule \eqref{eq:leja.def} then leads to
\[
  y_{(0)}=-1,\quad 
  y_{(1)}=1,\quad
  y_{(2)}=0, \quad
  y_{(3)} \approx -0.57735,\quad
  y_{(4)} \approx 0.65871, ~\ldots ~.
\]
Another common sequence, referred to as \emph{R-Leja} (real Leja) points, is obtained by 
carrying out the Leja construction on the upper unit circle in the complex plane in place of $\Gamma=[-1,1]$ and then projecting the sequence thus obtained onto the real line.
This results in (see e.g. \cite{Chkifa2013} for a proof):
\begin{gather*}
   y_{(i)} = \cos \phi_{(i)}, \quad i \in \bbN_0, \\
  \phi_{(0)}=0, \quad \phi_{(1)}=\pi, \quad \phi_{(2)}=\pi/2, \quad 
  \phi_{(2n+1)} = \frac{\phi_{(n+2)}}{2}, \quad \phi_{(2n+2)} = \phi_{(2n+2)} +\pi.  
\end{gather*}
For both Leja and R-Leja nodes, we may utilize any strictly increasing growth function $\m$ with $\m(0)=0$ to construct nested node sets $\mc Y_k\subset\mc Y_{k+1}$ as in \eqref{eq:node_sets}. 
The most common choice uses sets growing in unit increments, i.e., $\m(i) = i$.

Besides the Leja construction, \emph{Clenshaw--Curtis} nodes are also popular collocation points. 
Here, the node sets $\mc Y_k$ consist of the extrema of Chebyshev polynominals
\[
	\mc Y_0 = \{0\}, \qquad
	\mc Y_k = \left\{ - \cos\left( \pi i \ / \m(k) \right) \colon i = 0,\ldots,\m(k) \right\}, \quad k \in \bbN.
\]
Nestedness of the $\mc Y_k$ is then achieved by the \emph{doubling rule} $\m(k) = 2^{k}$ for $k\geq 1$.
The corresponding sequence of nodes $(y_{(i)})_{i\in\bbN_0}$ is given, suitably arranged, by \begin{align*}
	&y_{(0)} = 0,\\
	&y_{(1)} = - \cos\left( 0 \right),
	&&y_{(2)} = - \cos\left( \pi \right),\\
	&y_{(3)} = - \cos\left( 1/4 \pi\right),
	&&y_{(4)} = - \cos\left( 3/4 \pi \right),
	\ldots
\end{align*}

\paragraph{Sparse collocation}

We consider \emph{hierarchical} sparse collocation based on nested sequences of node sets $\mc Y_k$ as introduced above.
Let $\mc P_k(\Gamma)$ denote the set of univariate polynomials on $\Gamma$ of degree at most $k\in\bbN_0$.
We can then define for any Hilbert space-valued continuous function $f \colon \Gamma \to \mc H$ two objects:
\begin{itemize}
\item 
a Lagrange interpolant $\mc I_k \colon  C(\Gamma; \mc H) \to  \mc P_{\m(k)}(\Gamma; \mc H)$,
\item 
a univariate \emph{detail operator} 
$\Delta_{k} \colon C(\Gamma; \mc H) \to \mc P_{\m(k)}(\Gamma; \mc H)$,
\[
	\Delta_0 = \mc I_0, 
	\qquad 
	\Delta_{k} := \mc I_{k} - \mc I_{k-1},
	\quad
	k\in\bbN.
\]
\end{itemize}
With these definitions, we have that 
\begin{equation}\label{eq:Delta_P}
  \Delta_i f = 0 \qquad \forall f \in \mc P_{k}(\Gamma,\mathcal{H}),  \quad \forall i > \m^{-1}(k).
\end{equation}
Since $\Delta_k f = \mc I_{k}f - \mc I_{k-1}f = \mc I_{k}(f - \mc I_{k-1}f)$,
and due to the nestedness of the node sets $\mc Y_{k-1} \subset \mc Y_k$, the detail operators may be expressed as
\begin{align*}
	\Delta_k f
	&=
	\sum_{i = \m(k-1)+1}^{\m(k)}
	\left[
	f(y_{(i)}) - \mc I_{n-1}f(y_{(i)})
	\right]
	\ell^{(\m(k))}_i, \\
	\ell_{i}^{(\m(k))}(y)
	& :=
 	 \prod_{j = 0, j \neq i}^{\m(k)} \frac{y - y_{(j)}}{y_{(i)} - y_{(j)}}
  	\in
 	\mc P_{\m(k)} \quad \rev{\mbox{for } i  \in \{\m(k-1) +1, \ldots, \m(k)\}}.
\end{align*}
It is therefore convenient to introduce the notation      
\begin{equation}\label{eq:hier-lagr}
	h_i(y)
	:=
	 \ell_{i}^{(\m(k))}(y),
	 \qquad y\in\Gamma,
\end{equation}
where $i \in \{\m(k-1)+1, \ldots, \m(k)\}$.
The polynomials $h_i$, each associated to a node $y_{(i)}$, $i\in\bbN_0$, are called
\emph{hierarchical Lagrange polynomial}%
\footnote{\rev{The difference from the standard Lagrange polynomials is that $h_i$ is only defined for the nodes $y_{(i)}$ most recently added, with $i  \in \{\m(k-1) +1, \ldots, \m(k)\}$, whereas the standard Lagrange polyomials are redefined for all $i \in \{1, \ldots, \m(k)\}$ when new nodes are added.}},
$h_i\in \mc P_{\m(k)}$.
The \rewrite{term}{quantity} 
$
f(y_{(i)})-\mc I_{n-1}f(y_{(i)})=(f-\mc I_{n-1}f)(y_{(i)})
$
is also called \emph{hierarchical surplus}.
Next, consider tensorized detail operators 
\[
	\Delta_{\vi} := \bigotimes_{m=1}^M \Delta_{i_m},
	\qquad
	\Delta_{\vi}\colon C(\vGamma; \mc H) \to \mc P_{\m(\vi)}(\vGamma; \mc H),
\]
where $\m(\vi) = (\m(i_1),\ldots,\m(i_M)) \in \bbN^M$ and
\[
	\mc P_{\m(\vi)}
	=
	\mathrm{span}\{\vy^\vj \colon j_m \leq \m(i_m) \ \text{ for } m=1,\ldots,M\}.
\]
Given a (finite) subset $\Lambda \subset \F$ we define the \emph{sparse grid collocation operator} associated with the \emph{sparse grid} $\mc Y_\Lambda$ by
\[
	S_\Lambda := \sum_{\vi \in \Lambda} \Delta_\vi,
	\qquad
	\mc Y_{\Lambda} := \bigcup_{\vi \in \Lambda} \mc Y_{\vi},
	\qquad
	\mc Y_{\vi}
	:= \mathcal{Y}_{i_1} \times \mathcal{Y}_{i_2} \times \ldots \times \mathcal{Y}_{i_M}.
\]
We require the multi-index sets $\Lambda\subset \F$ to be \emph{downward-closed} (or \emph{monotone}), which means that
$
	\vi \in \Lambda
	\text{ implies } 
	\vi - \ve_m \in \Lambda,
$ 
where $\ve_m$ denotes the $m$th canonical unit multi-index. 
Downward-closedness of $\Lambda$ implies three facts (see e.g.\ \cite{ErnstEtAl2018}):
First, 
  \[
    \mc Y_{\Lambda} = \left\{\vy_{(\vj)}\colon \vj \leq \m(\vi), \ \vi \in \Lambda\right\},
    \qquad
    \vy_{(\vj)} := (y_{(j_1)}\,\,y_{(j_2)}\,\, \cdots \,\,y_{(j_M)})  \in \vGamma,
  \]  
  where $\vj \leq \m(\vi)$ is understood componentwise;
second, that the sparse grid collocation operator yields an approximation in $\mc P_{\m(\Lambda)}(\vGamma;\mathcal{H})$,
\[
  S_\Lambda\colon C(\vGamma;\mathcal{H}) \to \mc P_{\m(\Lambda)}(\vGamma;\mathcal{H}),
  \qquad
  \m(\Lambda)
  :=
  \{\vj \in \F \colon \vj \leq \m(\vi) \text{ for \rewrite{a}{some} } \vi \in \Lambda \};
\]
and third, together with the nestedness of the node sets, that $S_\Lambda$ is \emph{interpolatory}, i.e., 
\[
	S_\Lambda f(\vy_{(\vi)}) = f(\vy_{(\vi)}) \qquad \forall \vy_{(\vi)} \in \mc Y_\Lambda.
\]
\begin{remark}
For finite and monotone multi-index sets $\Lambda$ there exists $N \in \bbN$ multi-indices $\vi_1,\ldots,\vi_N \in \Lambda$ such that
\[
	\Lambda = \bigcup_{n=1}^J \mc R_{\vi_n},
	\qquad
	 \mc R_{\vi}
	 :=
	 \{
	 \vj \in \F\colon \vj\leq \vi
	 \},
\]
i.e., the multiindices $\vi_n$ can be viewed as the \emph{corners} of $\Lambda$. 
As an immediate consequence, we have
\[
	\mc P_{\m(\Lambda)}(\vGamma;\mathcal{H})
	=
	\bigoplus_{n=1}^N \mc P_{\m(\vi_n)}(\vGamma;\mathcal{H}).
\] 
\end{remark}

\paragraph{Adaptive sparse collocation algorithms}
Two ways to construct monotone multi-index sets $\Lambda$ for (hierarchical) sparse grid collocation
are the classical algorithm introduced by Gerstner and Griebel in \cite{GerstnerGriebel2003}
(as well as numerous variations mentioned in the literature surveyed in the introduction)
and the alternative algorithm introduced by Guignard and Nobile in \cite{GuignardNobile2018}.
Both can be seen as specific instances of the generic Algorithm~\ref{alg:aaa}.
We describe the former here and the latter (or rather, a slight variation thereof) in the next section, together with a convergence analysis.
To introduce these algorithms, we need to specify three ``ingredients'': 
the candidate set $\Cand_n$,
a \emph{marking strategy} for determining marked sets $\Mark_n \subset \Cand_n$, 
and corresponding estimates $\eta_n(\vk)$ for the error contribution of indices in the candidate set.
To this end, we require the following definitions \rev{(see also Figure \ref{fig:sets})}:
\begin{itemize}
\item 
The \emph{margin} $\Marg{\Lambda} \subset \F$ of a multi-index set $\Lambda \subset \F$ is given by
\[
   \Marg{\Lambda} 
   := 
   \{ \vk \in \F\setminus \Lambda
      \colon \vk - \ve_m \in \Lambda \text{ for some } m \in \bbN \}.
\]
\item 
The \emph{reduced margin} $\RMarg{\Lambda} \subset \Marg{\Lambda}$ of a subset $\Lambda \subset \F$ is given by
\[
   \RMarg{\Lambda} 
   := 
   \{ \vk \in \Marg{\Lambda}
      \colon \vk - \ve_m \in \Lambda \text{ for all } m \in \bbN \}.
\]
\item 
The \emph{monotone envelope} $E_\Lambda(\vk)\subset \Marg{\Lambda}$ of a multi-index $\vk \in \Marg{\Lambda}$:
\begin{equation}\label{eq:envelope}
   E_\Lambda(\vk) 
   := 
   \bigcap 
   \{ E \subset \Marg{\Lambda} 
      \colon \vk \in E \text{ and } \Lambda \cup E \text{ is monotone}  \}.  
\end{equation}
Note that $E_\Lambda(\vk) \rev{\cup \Lambda}$ is the smallest (in cardinality) monotone multi-index set containing $\Lambda \cup \{\vk\}$
and that for $\vk \in \RMarg{\Lambda}$ we have $E_\Lambda(\vk) = \{\vk\}$ by construction. 
\end{itemize}
The adaptive procedure in \cite{GerstnerGriebel2003} now chooses 
\begin{itemize}
\item 
as \rewrite{its}{} candidate set \rev{$\Cand_n$} the reduced margin of $\Lambda_n$, i.e.\ $\Cand_n = \RMarg{\Lambda_n}$;
\item 
\rewrite{and estimates}{as estimators $\eta_n$,} approximating the error contribution of $\vk \in \Cand_n$ by the $L^p$-norm of the hierarchical surplus, i.e.,
\begin{equation} \label{equ:error_indicator}
     \eta_n(\vk) = \|\Delta_\vk u \|_{L^p_\mu(\vGamma; \mc H)}, 
     \quad \vk \in \RMarg{\Lambda_n}.      
\end{equation}
Note that this is merely an error \emph{indicator} and not a proper \emph{estimator}, i.e.,
no proof of the properties required by Theorem~\ref{theo:abstract_adapt} is available.
A large body of literature, however, provides numerical evidence that this error indicator is quite robust and gives good results in practice;
\item 
\rewrite{its marking strategy selects an}{as marking strategy, to select the} index in the reduced margin which maximizes the value of $\eta_n$, i.e.,
   $\Mark_n = \{ \argmax_{\vk \in \RMarg{\Lambda_n}} \eta_n(\vk) \}$. 
An alternative strategy would be to use D\"orfler marking and mark e.g.\ the 50\% of the indices in the reduced margin with the largest $\eta_n$, cf.\ \cite{doerfler:marking}.
\end{itemize}
\begin{figure}[tbp]
  \centering
  \includegraphics[width=0.5\linewidth]{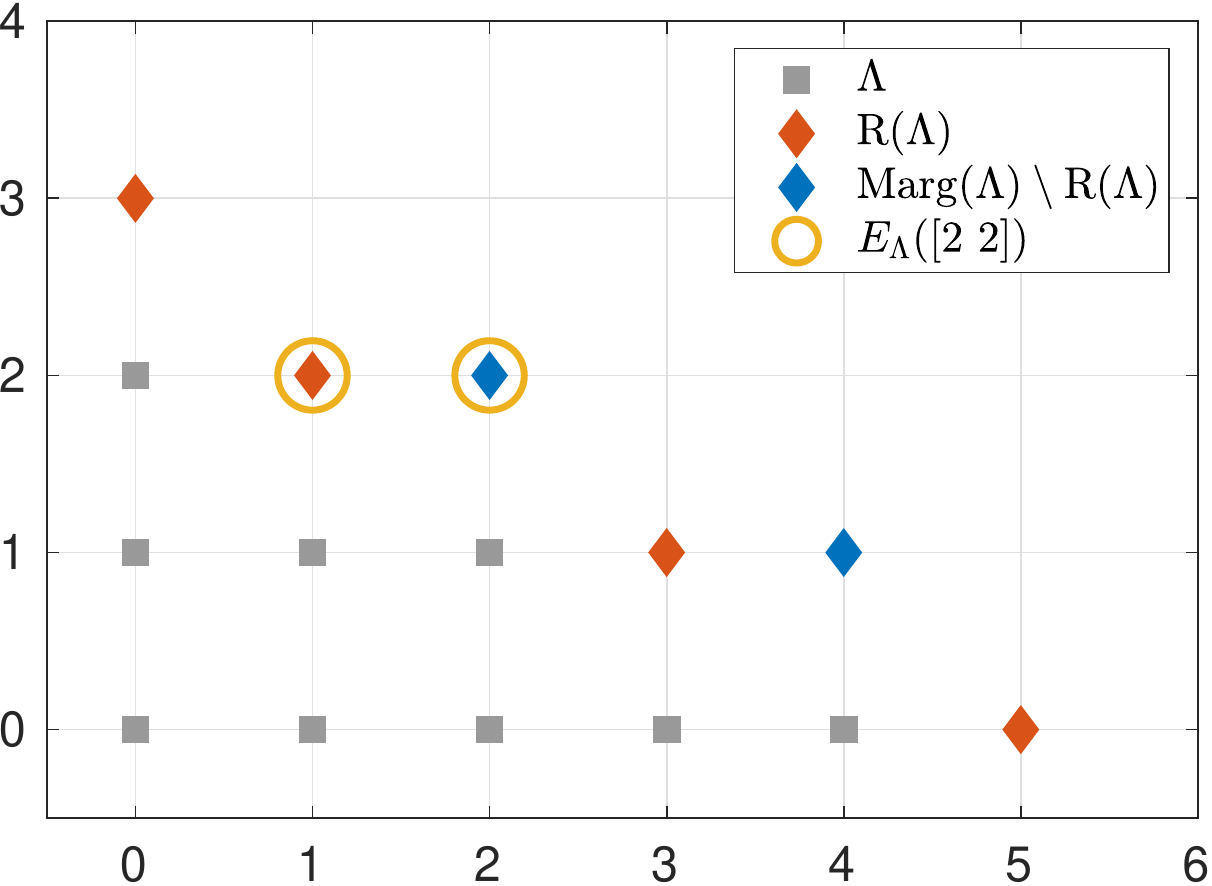}
  \caption{\rev{A multi-index set $\Lambda \subset \mathbb N_0^2$ (gray squares) and its margin $\Marg{\Lambda}$ (colored diamonds): more specifically, the multi-indices
      of $\Marg{\Lambda}$ that also belong to the reduced margin $\RMarg{\Lambda}$ are colored in red, whereas the remaining ones are colored in blue.
      Finally, we mark with yellow circles the indices of $\Marg{\Lambda}$ that constitute $E_\Lambda([2,\,2])$, i.e. the monotone envelope of $\vk=[2, \,2]$.}}
  \label{fig:sets}
\end{figure}
Algorithm~\ref{alg:GGadap} summarizes the Gerstner--Griebel scheme as pseudocode.
\begin{algorithm}[t]
 \begin{algorithmic}[1]
   \STATE $\Lambda_0 := \{\boldsymbol 0\}$
   \STATE $u_0 := S_{\Lambda_0}u$
   \FOR{$n\in\bbN_0$}
   \STATE Compute reduced margin $\RMarg{\Lambda_n}$    
   \STATE Compute error indicators (reduced margin):
   \[
    \rev{\eta_n(\vk) = \|\Delta_\vk u \|_{L^p_\mu(\Gamma; \mc H)}, \quad \vk \in \RMarg{\Lambda_n}}
   \]
   \STATE Choose $\vk^*_n := \argmax_{\vk \in \RMarg{\Lambda_n}} \eta_n(\vk) $
   \STATE Set $\Lambda_{n+1} := \Lambda_n \cup \{ \vk^*_n \}$
          and $u_{n+1} := S_{\Lambda_{n+1}} u$. 
   \ENDFOR
\end{algorithmic}\caption{Adaptive sparse grid algorithm of Gerstner and Griebel  \cite{GerstnerGriebel2003}} \label{alg:GGadap}
\end{algorithm}
Note that, since $S_\Lambda$ is interpolatory for $\mc Y_n$ nested and  $\Lambda$ monotone,
we can efficiently compute $\eta_n$ in \eqref{equ:error_indicator}, and therefore $S_{\Lambda_{n+1}}$ based on $S_{\Lambda_n}$.
For this, let $\vi \in \RMarg{\Lambda_n}$ and $\Lambda_{n+1} = \Lambda_n \cup \{\vi\}$.
Then,
\[
	\Delta_{\vi} u
	=
	\sum_{\vy_{(\vj)} \in \mc Y_{\vi}\setminus \mc Y_{\Lambda} }
	[u(\vy_{(\vj)}) - (S_{\Lambda_n}u)(\vy_{(\vj)})] \, h_\vj,
	\qquad
	h_\vj(\vy):=\prod_{m=1}^M h_{j_m}(y_m),
\]
where the $h_i$ are the univariate hierarchical Lagrange polynomials defined in (\ref{eq:hier-lagr}) and the set of additional nodes $\mc Y^+_{\vi} := \mc Y_{\vi}\setminus \mc Y_{\Lambda}$ is
\[
	\mc Y^+_{\vi}
	= \mathcal{Y}^+_{i_1} \times \mathcal{Y}^+_{i_2} \times \ldots \times \mathcal{Y}^+_{i_M},
	\qquad
	\mathcal{Y}^+_{i}
	:=
	\mathcal{Y}_{i}\setminus \mathcal{Y}_{i-1}
	=
	\left\{ y_{(j)} \colon \m(i-1)+1 \leq j \leq \m(i) \right\}.
\]
The main shortcoming of this approach is that the computation of $\Delta_{\vi} u$ requires solving the PDE to evaluate $u(\vy_{(\vi)})$, and for this reason one may refer to this algorithm as \emph{fully a posteriori}. \rev{Clearly, it would be a waste of computational resources to discard these additional PDE solutions: therefore, practical implementations
  of Algorithm \ref{alg:GGadap} ultimately augment $\Lambda$ to $\Lambda_{\text{end}} = \Lambda_n \cup \RMarg{\Lambda_n}$ at the last iteration and return
  $u_{\text{end}} = S_{\Lambda_{\text{end}}}$ instead of $S_{\Lambda_{n}}$. Nonetheless, this procedure is ``suboptimal'' in terms of computational effort.} 
If the reduced margin is large, this operation can be expensive. 
Moreover, as previously mentioned, the choice of $\eta_n$ in  \eqref{equ:error_indicator} is a heuristic and no convergence proof for the adaptive algorithm is available. 
To overcome this issue, we introduce and analyze in the next section another variation of Algorithm \ref{alg:aaa}, for which we can prove convergence.

We close this section by pointing out that using a hierarchical basis is convenient but not necessary, and the standard (non-hierarchical) Lagrange basis can also be used to implement Algorithm~\ref{alg:GGadap}. To this end, 
one would need to draw on the so-called \emph{combination technique} \cite{Griebel.Schneider.Zenger:1992} for evaluating the detail operators $\Delta_{\vi} u$
as a linear combination of tensorized Lagrange interpolants, 
  \[
    \Delta_{\vi} u = \sum_{\vj \in \{0,1\}^M} (-1)^{|\vj|} (\mathcal{I}_{i_1 - j_1} \otimes \mathcal{I}_{i_2 - j_2} \otimes \cdots \otimes \mathcal{I}_{i_M - j_M})u,
  \]
and to adjust the computation of $S_{\Lambda}u$ accordingly, see e.g.\ \cite{NobileEtAl2016,GuignardNobile2018}; this has the advantage that non-nested sequences of node sets (such as zeros of orthogonal polynomials) can be used if desired, see e.g. \cite{NobileEtAl2016,ErnstEtAl2018}.

\section{Adaptive Sparse Collocation for the Diffusion Problem}
\label{sec:convergence}

We now turn attention to \rewrite{the}{our} above-mentioned slight variation of the adaptive algorithm by Guignard and Nobile from \cite{GuignardNobile2018}; see Remark~\ref{rem:AA_GN} for a discussion on the difference between the two versions.
This algorithm is based on the following error estimator, for which reliability has been established in \cite{GuignardNobile2018}.
\begin{proposition}[{\cite[Proposition 4.3]{GuignardNobile2018} }] \label{propo:diane}
  Let $u$ denote the solution of the random elliptic PDE given in equation \eqref{eq:bvp} with linear diffusion coefficient
    as in \eqref{equ:affine-a}, and let $\Lambda\subset\F$ be a monotone subset such that the sparse
  grid collocation operator $S_\Lambda$ as introduced in Section \ref{sec:collocation} is interpolatory.
  Then, for any $p\in[1,\infty]$ we have
\[
	\|u - S_\Lambda u\|_{L^p_\mu(\vGamma; H_0^1(D))}
	\leq
	\frac1{a_{\min}}
	\sum_{\vk \in \Marg{\Lambda}} \|\Delta_\vk (a \nabla S_\Lambda u) \|_{L^p_\mu(\vGamma; L^2(D))}.
\]
\end{proposition}


This proposition suggests  
$\eta_n(\vk) := \|\Delta_\vk (a \nabla S_{\Lambda_n} u) \|_{L^p_\mu(\vGamma; L^2(D))}$ 
as an error estimator for adaptively constructing the sparse grid approximations 
$u_n = S_{\Lambda_n}u_n$ 
and also to consider the entire margins $\Marg{\Lambda_n}$ as candidate sets.
This yields Algorithm \ref{alg:asc}.
Note here that the value $p\in[1,\infty]$ has to be chosen in advance and that $\Cand_n := \Marg{\Lambda_n}\subset \F$ is, in fact, finite for finite $M$.
Moreover, we highlight that Proposition \ref{propo:diane} implies that Algorithm \ref{alg:asc} satisfies the first assumption (reliable error estimator) of the abstract convergence result, \rev{stated in} Theorem \ref{theo:abstract_adapt}.
Besides that, also the third assumption of Theorem \ref{theo:abstract_adapt} is satisfied by construction, i.e, by the marking strategy $\Mark_{n} := E_{\Lambda_n}(\vk^*_n)$
\rev{(where $E_{\Lambda_n}(\vk^*_n)$ is the monotone envelope of $\Lambda_n$, see Equation \eqref{eq:envelope})}
and the choice of $\vk^*_n$, cf. Remark \ref{rem:abstract_theorem}.

\begin{algorithm}
  \caption{Adaptive sparse grid algorithm for the diffusion problem \eqref{eq:bvp}, \rev{variation of Guignard--Nobile in \cite{GuignardNobile2018}}} \label{alg:asc}
 \begin{algorithmic}[1]
  \STATE $\Lambda_0 := \{\boldsymbol 0\}$
  \STATE $u_0 := S_{\Lambda_0}u$
  \FOR{$n\in\bbN_0$}
    \STATE Compute margin as candidate set $\Cand_n := \Marg{\Lambda_n}$\\
    \STATE Compute error estimators:
    \begin{equation}\label{equ:error_estimator}
      \eta_n(\vk) := \|\Delta_\vk (a \nabla u_n) \|_{L^p_\mu(\vGamma; L^2(D))}, \quad \vk \in \Marg{\Lambda_n}     
    \end{equation}

    \STATE choose $\vk^*_n := \argmax_{\vk \in \Cand_n} \eta_n(\vk)$\\
    \STATE set $\Mark_{n} := E_{\Lambda_n}(\vk^*_n)$\\
    \STATE set $\Lambda_{n+1} := \Lambda_n \cup \Mark_n$\\
    \STATE compute $u_{n+1} := S_{\Lambda_{n+1}} u$. 
    \ENDFOR
    \end{algorithmic}
\end{algorithm}


\begin{remark}[Adaptive algorithm in \cite{GuignardNobile2018}]\label{rem:AA_GN}
  The difference between Algorithm \ref{alg:asc} and \rewrite{the}{its original} version \rev{by Guignard and Nobile} in \cite{GuignardNobile2018}
  is that in \cite{GuignardNobile2018} the following \emph{profit indicators} are introduced instead of the error estimator $\eta_n(\vk)$ given in \eqref{equ:error_estimator}:
 \begin{equation}\label{equ:AA_GN_profits}
      \pi_n(\vk) := \frac{\sum_{\vi \in E_{\Lambda_n}(\vk)} \eta_n(\vi)}{\sum_{\vi \in E_{\Lambda_n}(\vk)} W(\vi)}, \qquad \vk \in \Marg{\Lambda_n},   
    \end{equation}
with $W(\vi)$ denoting the work contribution of the multi-index $\vi$, i.e., the number of new grid points in $\mc Y^+_\vi$ required to evaluate $\Delta_\vi$ which is given by
\[
	W(\vi) := |\mc Y_\vi^+| = \prod_{m=1}^M (\m(i_m) - \m(i_m-1)).
\]
Then, $\vk^*_n$ is chosen as
\begin{equation}\label{eq:AA_GN_marking}
	\vk^*_n := \argmax_{\vk \in \Cand_n} \pi_n(\vk), \qquad \Mark_{n} := E_{\Lambda_n}(\vk^*_n).
\end{equation}
In the case of linearly growing univariate node sets $\m(i) = i$ we have $W(\vi)\equiv 1$, i.e., $\pi_n(\vk) = \frac1{|E_{\Lambda_n}(\vk)|} \sum_{\vi \in E_{\Lambda_n}(\vk)} \eta_n(\vi)$ corresponds to the average error estimator on the monotone envelope $E_{\Lambda_n}(\vk)$.
We provide a more detailed discussion on both versions of the adaptive algorithm for the elliptic problem in Section \ref{sec:numerics} with a focus on computational aspects.
\end{remark}

We now turn to our main result stating the convergence of Algorithm \ref{alg:asc}, under rather mild assumptions on the employed univariate interpolation nodes. 
Specifically, we assume an algebraic growth of the operator norm of the associated detail operators
\begin{equation} \label{equ:detail_opnorm}
	\|\Delta_k \|_\infty 
	:= 
	\sup_{0 \not\equiv f \in C(\Gamma; \bbR)} 
	\frac{ \|\Delta_k f\|_{C(\Gamma;\bbR)} }
	     { \|f\|_{C(\Gamma;\bbR)} },		\qquad k\in\bbN_0.
\end{equation}

\begin{theorem}[Convergence of Algorithm \ref{alg:asc}]\label{theo:GN_conv}
Given the assumptions of Theorem \ref{cor:Bachmayr_finite} and assuming there exist finite constants $0\leq c, \theta < \infty$ such that
\begin{equation}\label{equ:Delta_Lebesgue}
	\|\Delta_k \|_\infty 
	\leq (1 + ck )^\theta \qquad
	\forall k\in\bbN_0,
\end{equation}
the approximations $u_n$ constructed by Algorithm \ref{alg:asc} 
satisfy 
\[
	\lim_{n\to\infty} \|u - u_n\|_{L^p_\mu(\vGamma; H_0^1(D))} = 0.
\]
\end{theorem}

We already established above that Algorithm \ref{alg:asc} satisfies the first and third assumption of the abstract convergence theorem,
\rev{i.e.} Theorem \ref{theo:abstract_adapt}.
It thus remains to verify the second assumption.
This turns out to be rather technical and is presented in detail in Section \ref{sec:proof}.

We now comment on the additional assumption \eqref{equ:Delta_Lebesgue} of Theorem \ref{theo:GN_conv} regarding the operator norms $\|\Delta_k \|_\infty$ of the univariate detail operators.
Condition \eqref{equ:Delta_Lebesgue} is rather mild and satisfied, e.g., if the corresponding interpolation operators $\mc I_k$
  possess an at most algebraically increasing Lebesgue constant:
\begin{equation}\label{equ:I_Lebesgue}
	\|\mc I_k \|_\infty := \sup_{f \colon \|f\|_{C(\Gamma;\bbR)}=1} \|\mc I_k f\|_{C(\Gamma;\bbR)}
	\leq c_1 + c_2 n^\theta
	\qquad 
	\forall k\geq 1,
\end{equation}
for finite constants $0\leq c_1,c_2,\theta < \infty$, since then with a finite $c  = c(c_1,c_2,\theta) < \infty$
\[
	\|\Delta_k \|_\infty \leq \|\mc I_k \|_\infty + \|\mc I_{k-1} \|_\infty \leq 2c_1 + 2c_2 k^\theta \leq c k^\theta
	\qquad
	\forall k\geq 1,
\]
and $\Delta_0 = \mc I_0$, i.e., $\|\Delta_0 \|_\infty = \|\mc I_0 \|_\infty = 1$.
Note that the algebraic growth bound \eqref{equ:I_Lebesgue} holds, for instance, for interpolation based on Leja and R-Leja nodes $y_{(j)} \in [-1,1]$ introduced above, see \cite{Chkifa2013,Chkifa2015} and references therein, where such bounds were proved for Leja and R-Leja nodes, respectively:
  \[
    \|\mc I_k\|_{\infty} \leq 5k^2 \log k, \mbox{ for } k \geq 2, \qquad 
    \|\mc I_k\|_{\infty} \leq 2k, \mbox{ for } k \geq 1.
  \]
Moreover, for Clenshaw--Curtis nodes combined with the doubling rule $\m(k) = 2^k$, $k\geq 1$, we obtain by classical results \cite{McCabePhillips1973, Brutman1978} that
\[
	\|\mc I_k\|_{\infty}
	\leq
	1 + \frac 2\pi \log\left( \m(k) \right)
	=
	1 + \frac {2\log 2}\pi \ k,
	\qquad
	k\geq 1.
\]


\subsection{Extensions of Theorem  \ref{theo:GN_conv}}

In this subsection we comment on two possible extensions of our convergence analysis. 


\paragraph{Convergence of the adaptive algorithm \rev{by Guignard and Nobile} in \cite{GuignardNobile2018}}
As outlined in Remark \ref{rem:AA_GN}, the adaptive algorithm proposed \rev{by Guignard and Nobile} in \cite{GuignardNobile2018} differs from Algorithm \ref{alg:asc} only in the marking strategy or, to be more precise, by the choice of $\vk^*_n$, see \eqref{eq:AA_GN_marking}.
Thus, in order to extend Theorem \ref{theo:GN_conv} to this algorithm it suffices to verify that the third assumption
of Theorem \ref{theo:abstract_adapt} also holds for the marking strategy \eqref{eq:AA_GN_marking} w.r.t.~to the error estimators $\eta_n$ given in \eqref{equ:error_estimator}. 
We focus on the case of Leja nodes with a linear growth function $\m(i) \equiv i$ here, since the \rewrite{adaptive algorithm by \cite{GuignardNobile2018} using}{the version with} Clenshaw--Curtis nodes was analyzed in the recent \rewrite{work}{work on convergence} \cite{Feischl2020} \rev{mentioned in the introduction}.
\rev{If Leja points are considered,} we can easily \rewrite{then}{} ensure convergence by a mild additional assumption: there exists a constant  $0< c < \infty$ such that for any monotone multi-index set $\Lambda$ we have
\begin{equation}\label{eq:AA_GN_cond}
	\max_{\vk \in \Marg{\Lambda}}
	\eta_{\Lambda}(\vk)
	\leq
	c\  
	\max_{\vk \in \RMarg{\Lambda}}
	\eta_{\Lambda}(\vk),
	\qquad
	\eta_{\Lambda}(\vk)
	:=
	\|\Delta_\vk (a \nabla S_\Lambda u) \|_{L^p_\mu(\vGamma; L^2(D))},
\end{equation}
i.e., the largest error estimator in the \emph{full margin} can be bounded by the constant times the largest error estimator in the \emph{reduced margin}.
Indeed, by construction of the profits $\pi_n$ in \eqref{equ:AA_GN_profits} and of the marking strategy in \eqref{eq:AA_GN_marking} we have for $\m(i)\equiv i$ that $\pi_n(\vk) = \eta_n(\vk)$ if $\vk \in \RMarg{\Lambda_n}$ and
\[
	\max_{\vk \in \rev{\RMarg{\Lambda_n}}}
	\eta_{n}(\vk)
	=
	\max_{\vk \in \rev{\RMarg{\Lambda_n}}}
	\pi_{n}(\vk)
	\leq
	\frac{\sum_{\vi \in \Mark_{n}} \eta_n(\vi)}{\sum_{\vi \in \Mark_{n}} W(\vi)}
	\leq
	\sum_{\vi \in \Mark_{n}}
	\eta_n(\vi).
\]
Hence, condition \eqref{eq:AA_GN_cond} then guarantees that the third assumption of Theorem \ref{theo:abstract_adapt} is also satisfied for the marking strategy \eqref{eq:AA_GN_marking}.
We consider \eqref{eq:AA_GN_cond} as a plausible assumption in practice, although pathological counterexamples may possibly be constructed.

\paragraph{Convergence of the Gerstner--Griebel algorithm}
The abstract convergence result, Theorem \ref{theo:abstract_adapt}, as well as our techniques for proving Theorem \ref{theo:GN_conv} can also be exploited to show convergence of \rev{the adaptive algorithm by Gerstner and Griebel in \cite{GerstnerGriebel2003}, i.e. of} Algorithm \ref{alg:GGadap}.
To this end, \rewrite{we require, of course,}{we need of course to assume} the reliability of the error indicators $\eta_n(\vk) = \|\Delta_\vk u\|_{L^p_\mu(\vGamma;\mc H)}$.
Since these hierarchical surpluses are not connected to the model problem \eqref{eq:bvp},
as is the case for the residual-based error estimators \eqref{equ:error_estimator}, we state the Theorem in a more general setting,
i.e., we consider general Hilbert space-valued mappings $u\colon \vGamma \to \mc H$ and
moreover, we do not restrict to solutions $u$ that admit a Taylor expansion,
but rather consider the more general case of a solution that admits an expansion
over polynomials $P_k$ with a certain growth of their maximum norm.
\rev{Reliability is also not proved here but merely assumed, and must be checked on a case-by-case basis.}
\begin{theorem}[\rev{Convergence of Algorithm \ref{alg:GGadap} by Gerstner and Griebel, \cite{GerstnerGriebel2003}}]\label{theo:GG_alg}
Let $\mc H$ be a separable Hilbert space and let $u \in C(\vGamma;\mc H)$ allow for a polynomial expansion \eqref{equ:Poly_Exp} converging in $L^p_\mu(\vGamma;\mc H)$ for a $p \in [1,\infty]$ where the corresponding univariate polynomials $P_k \in \mc P_k(\Gamma;\bbR)$ satisfy 
\begin{equation}\label{equ:poly_cond}
	\| P_k \|_{C(\Gamma; \bbR)} 
	\leq 
	(1 + \widetilde{c} k)^{\widetilde{\theta}}
\end{equation}
for finite constants $\widetilde{c}, \widetilde{\theta} \geq 0$.
Further assume that
\begin{enumerate}
\item
the coefficients $u_\vk \in \mc H$, $\vk \in \F$, of the polynomial expansion \eqref{equ:Poly_Exp} satisfy
\[
	\left(\vrho^\vk \|u_\vk\|_{\mc H}\right)_{\vk \in \F} \in \ell^2(\F)
\]
for a weight vector $\vrho \in\bbR^M$ with $1 < \rho_m$ for all $m=1,\ldots,M$;
\item
there exists a constant $C < \infty$ such that for any finite and monotone $\Lambda \subset \F$
\begin{equation}\label{eq:surplus_condition}
	\|u - S_\Lambda u\|_{L^p_\mu(\vGamma;\mc H)}
	\leq
	C \sum_{\vk \in \RMarg{\Lambda}}
	\|\Delta_\vk u\|_{L^p_\mu(\vGamma;\mc H)};
\end{equation}
\item
the univariate detail operators $\Delta_k$ satisfy \eqref{equ:Delta_Lebesgue} for finite constants $0\leq c, \theta < \infty$.
\end{enumerate}
Then we have for the approximations $u_n$ constructed by Algorithm \ref{alg:GGadap} that
\[
	\lim_{n\to\infty} \|u - u_n\|_{L^p_\mu(\vGamma; \mc H)} = 0.
\]
\end{theorem}
Note that the first item on the $u_\vk$ is satisfied for the model problem by Theorem \ref{cor:Bachmayr_finite} and that for Taylor polynomials condition \eqref{equ:poly_cond} holds with $\widetilde{c} = \widetilde{\theta} = 0$.
This theorem provides an overview of the three most important "ingredients" for convergence of adaptive collocation: exponentially decaying coefficients $u_\vk$, only algebraically growing norms of the $\Delta_\vk$ and reliability of the employed error indicators.
The proof of Theorem \ref{theo:GG_alg} is significantly easier than the proof of Theorem  \ref{theo:GN_conv}, because the error indicators do not depend on the current approximation. 
Nonetheless, proving Theorem \ref{theo:GG_alg} requires some auxiliary results stated in Section \ref{sec:proof} and is therefore postponed to Section \ref{sec:convergence proof_GG}.

\subsection{Computational Considerations} \label{sec:numerics}

Having established the convergence of \rewrite{the}{our variant of the algorithm by Guignard and Nobile, as stated in} Algorithm \ref{alg:asc},
as well as \rev{of} the Gerstner--Griebel adaptive sparse grid algorithm \rev{Algorithm \ref{alg:GGadap}} (GG algorithm for short in the following),
we comment on the computational advantages and disadvantages of both:
\begin{enumerate}
\item 
The GG algorithm considers candidate indices in the \emph{reduced margin} instead of the \emph{full margin}. 
This makes treating problems with high-dimensional parameters somewhat easier with the GG algorithm, since the size of the full margin grows substantially faster than the reduced margin.
\item 
However, as already noted, the GG algorithm is \emph{fully a posteriori}: evaluating the error indicators involves actually evaluating $u$ (i.e., solving additional PDEs) on the new collocation points $\mathcal{Y}^+_n(\vk) = \mathcal{Y}_{\vk} \setminus \mathcal{Y}_{\Lambda_n \cup \{\vk\}}$ for each $\vk \in \RMarg{\Lambda_n}$, see \eqref{equ:error_indicator} Algorithm~\ref{alg:GGadap}.
By contrast, Algorithm~\ref{alg:asc} computes its error estimator by evaluating \emph{the current sparse grid interpolant $u_n$} at the new collocation points $\mathcal{Y}^+_n(\vk)$ for $\vk \in \Marg{\Lambda_n}$.
This is a significant advantage of the error estimator-based algorithms
\rev{(both the original version by Guignard and Nobile and our variant Algorithm \ref{alg:asc})}
over the GG algorithm, in particular if solving the PDE for individual parameter values is computationally expensive
\rev{(even though these additional PDE solves are not discarded but ultimately enter the final approximation returned
  by Algorithm \ref{alg:GGadap}, as already discussed in Section \ref{sec:collocation})}.
\item 
On the other hand, because the error estimators are based on the current approximation, they have to be recomputed in each step of  Algorithm \ref{alg:asc}, i.e., in general $\eta_n(\vk) \neq \eta_{n+1}(\vk)$ for any $\vk \in \Marg{\Lambda_n} \cap \Marg{\Lambda_{n+1}}$. 
This is not required by the GG algorithm.
Thus, the evaluation of the sparse grid interpolant $u_n$ should be implemented in a very efficient way, since this operation is repeated at each iteration for an increasingly large number of multi-indices in the margin. 
In this sense, the hierarchical representation of the sparse grid interpolant via hierarchical Lagrange polynomials and hierarchical surpluses is to be preferred to the classical combination technique representation \cite{Griebel.Schneider.Zenger:1992}, 
since the former usually yields a faster evaluation---at the price of a higher offline-cost due to the computation of the surpluses.
\item 
The hierarchical sparse grid representation as well as the error estimators in \cite{GuignardNobile2018} for the diffusion problem require nested univariate node sets---for an efficient implementation and reliability, respectively.
Instead, the GG algorithm also works with non-nested nodes, see e.g. \cite{NobileEtAl2016,ErnstEtAl2018,ernst.eal:Leja+Levy}.
This might be a rather minor point, since suitable nested node families in form of Leja or Clenshaw-Curtis nodes are available.
\end{enumerate}

As an extensive numerical study of the error estimator-based adaptive scheme has been already carried out \rev{by Guignard and Nobile} in \cite{GuignardNobile2018},
we present no further numerical experiments here.
\rewrite{For instance, in [30] the authors}{In their study, they }
observed for several numerical test examples of the diffusion problem \eqref{eq:bvp} that the error estimator stated in Proposition~\ref{propo:diane} is sharp.
These test examples included different dimensions of the physical domain ($d=1,2$) as well as different numbers $M$ of parameter variables and different expansion functions $a_m$ in the definition of the diffusion coefficient.
Besides this, a second set of experiments in \cite{GuignardNobile2018} compared the performance of the error estimator-based algorithm and the GG algorithm: both showed a similar performance w.r.t.~the number of grid points in the corresponding adaptively constructed sparse grids $\mathcal Y_{\Lambda_n}$ (recall that each sparse grid point corresponds to a PDE solve); however, if all PDE solves (i.e., also those necessary for evaluating the profits on the margin) are taken into account, than the GG algorithm performed significantly less effectively.

Although \rewrite{the setting and}{} the algorithm \rev{by Guignard and Nobile} in \cite{GuignardNobile2018} slightly differs from \rewrite{the setting and}{} Algorithm~\ref{alg:asc} as considered here, these differences are negligible for the numerical performance for the following reasons:
\begin{itemize}
\item
 The version of Algorithm \ref{alg:asc} considered in \cite{GuignardNobile2018} considers normalized profit indicators $\pi_n$ for the indices $\vk$, see \eqref{equ:AA_GN_profits}. However, previous numerical evidence for the GG algorithm suggests that whether error indicators or profit indicators are used does not play a major role for the convergence, see e.g. \cite{NobileEtAl2016}.
Therefore, for the same reasons, one can expect Algorithm~\ref{alg:asc} to
exhibit similar numerical behavior as the \rev{original} adaptive algorithm \rev{by Guignard and Nobile} in \cite{GuignardNobile2018}.

\item 
Although the second set of results in \cite{GuignardNobile2018} is for Clenshaw--Curtis collocation points only, it is well-known that in practice the performance of Leja and Clenshaw--Curtis points is quite similar for adaptive sparse collocation using the GG algorithm, see e.g. \cite{nobile.etal:leja-points}.
Thus, it is again reasonable to assume that similar results to those reported in \cite{GuignardNobile2018} also hold for Algorithm~\ref{alg:asc} using Leja nodes.
  
\item 
The tests in \cite{GuignardNobile2018} are performed with $p=\infty$ only, both for the evaluation of the error and for the computation of the error indicator. 
Our theory covers any $p \in [1,\infty]$, and we expect that GG and Algorithm \ref{alg:asc} would behave similarly also for $p \neq \infty$. 
\end{itemize}

\section{Proofs \rewrite{of Convergence}{of Theorems \ref{theo:GN_conv} and \ref{theo:GG_alg}}}\label{sec:proof}
We \rewrite{collect}{begin this section by stating} four auxiliary results required for the subsequent proof of our main result\rev{s}, Theorem\rev{s} \ref{theo:GN_conv}
\rev{and \ref{theo:GG_alg}}.
First, we recall a statement on the operator norm of the tensorized detail operators $\Delta_\vi$ given in \eqref{equ:detail_opnorm}.

\begin{proposition}[{\cite[Section 3]{ChkifaEtAl2014}}]\label{propo:Delta_vi} 
For the operator norm \eqref{equ:detail_opnorm} of the tensorized detail operators
\[
	\|\Delta_\vi \|_\infty 
	= 
	\sup_{0\not\equiv f \in C(\vGamma; \bbR)} 
	\frac{\|\Delta_\vi f\|_{C(\vGamma;\bbR)}}
	     {\|f\|_{C(\vGamma;\bbR)}}, 
	\qquad \vi \in\F,
\]
there holds
\[
	\|\Delta_\vi \|_\infty = \prod_{m = 1}^M \|\Delta_{i_m} \|_\infty.
\]
\end{proposition}

\noindent Next, we provide an estimate for the sparse grid collocation operator $S_\Lambda$
applied to Taylor polynomials/multivariate monomials given an algebraically growing operator norm of the univariate detail operators.
This result is similar to~\cite[Proposition 3.1]{ErnstEtAl2018}.

\begin{proposition}
\label{propo:S_Tk}
Let there exist constants $1< c < \infty$ and $\theta<\infty$ such that
\[
	\|\Delta_i \|_\infty 
	\leq (1 + ci )^\theta, \qquad
	\forall i\in\bbN.
\]
Then for the Taylor polynomials $T_\vk(\vy) := \vy^\vk$, $\vk\in\F$, and $\vGamma = [-1,1]^M$ we have
\[
	\sup_{\Lambda \subseteq \F} \|S_\Lambda T_\vk\|_{C(\vGamma;\bbR)} 
	\leq
	\prod_{m = 1}^M (1+ c k_m)^{1+\theta},
	\qquad
	\vk \in\F.
\]
\end{proposition}
\begin{proof}
First, notice that with $\m^{-1}$ as in \eqref{eq:m_inverse} and using \eqref{eq:Delta_P} we have
\[
	\Delta_\vi T_\vk
	= \prod_{m=1}^M
	\Delta_{i_m} T_{k_m}
	\equiv 0
\]
if $i_m$ is such that $\m(i_m-1) \geq k_m$, i.e., if $i_m > \m^{-1}(k_m)$ for any $m$.
Thus, with $\mc R_\vk := \{\vj \in \F\colon j_m\leq k_m \ \forall m = 1,\ldots,M\}$, we obtain
\[
	\sup_{\Lambda \subseteq \F} \|S_\Lambda T_\vk\|_{C(\vGamma;\bbR)}
	=
	\max_{\Lambda \subseteq \mc R_{\m^{-1}(\vk)}} \|S_\Lambda T_\vk\|_{C(\vGamma;\bbR)},
\]
where $\m^{-1}(\vk) = (\m^{-1}(k_1),\ldots,\m^{-1}(k_M))\in\bbN_0^M$.
Moreover, the triangle inequality yields
\[
  \|S_\Lambda T_\vk\|_{C(\vGamma;\bbR)}
   \leq \sum_{\vi \in \Lambda} \|\Delta_\vi T_\vk\|_{C(\vGamma;\bbR)}
   \leq \sum_{\vi \in \Lambda} \|\Delta_\vi\|_\infty \ \|T_\vk\|_{C(\vGamma;\bbR)}
   \leq\sum_{\vi \in \Lambda} \prod_{m =1}^M (1 + ci_m )^\theta.
 \]
Since we are considering $\Lambda$ to be a subset of $\mc R_{\m^{-1}(\vk)}$, we can further bound the last term as follows
\[
\sum_{\vi \in \Lambda} \prod_{m =1}^M (1 + ci_m )^\theta
\leq\sum_{\vi \in \mc R_{\m^{-1}(\vk)}} \prod_{m = 1}^M (1 + ck_m )^\theta
\leq |\mc R_\vk| \prod_{m = 1}^M (1 + ck_m )^\theta
=\prod_{m = 1}^M (1 + ck_m )^{1+\theta},
\]
since $|\mc R_{\m^{-1}(\vk)}| \leq |\mc R_\vk| = \prod_{m = 1}^M (1 + k_m )$.
\end{proof}


Furthermore, we require a rather general result on the summability of sequences on $\F$.
\begin{lemma}[{\cite[Lemmas 2 and 3]{CohenMigliorati2018}}]\label{lem:cohen}
For any $0 < q < 1$, one has
\[
	\vrho \in \bbR^M \text{ and } \min_{m=1,\ldots,M}|\rho_m|  > 1
	\quad \Longleftrightarrow \quad
	\left( \vrho^{-\vk} \right)_{\vk\in\F} \in \ell^q(\F).
\]
Moreover, for any $0 < q < 1$ and any algebraic factor
\[
	\beta(\vk) := \prod_{m = 1}^M (1 + c k_m)^{\theta}, \qquad \vk\in\F,
\]
with finite $c, \theta \geq 0$, one has
\[
	\vrho \in \bbR^M \text{ and } \min_{m=1,\ldots,M}|\rho_m|  > 1
	\quad \Longleftrightarrow \quad
	\left( \beta(\vk) \ \vrho^{-\vk}\right)_{\vk\in\F} \in \ell^q(\F).
\]
\end{lemma}
Note that the original statement in \cite[Lemmas 2 and 3]{CohenMigliorati2018} is for the case of countable sequences $\vrho = (\rho_m)_{m\in\bbN} \in \ell^q(\bbN)$.

\bigskip

The last auxiliary result provides a simple estimate for the tails of converging series of the same form
$\left( \beta(\vk) \ \vrho^{-\vk}\right)_{\vk\in\F}$ as considered in the previous lemma.

\begin{proposition}\label{propo:tail_bound}
Let $\vrho \in \bbR^M$ be a vector of numbers $\rho_m>1$, $m=1,\ldots,M$, and 
\[
	\beta(\vk) := \prod_{m = 1}^M (1 + c k_m)^{\theta}, \qquad \vk\in\F,
\]
an algebraic factor with finite $c, \theta \geq 0$.
Then, we have for any $\vk \in \F$
\begin{equation}\label{equ:tail_bound}
	\sum_{\vj \geq \vk} \beta(\vj) \vrho^{-\vj} \leq 
	C\ \beta(\vk) \vrho^{-\vk},
	\qquad
	 C:= \sum_{\vk \in\F} \beta(\vk) \vrho^{-\vk} <\infty
\end{equation}

\end{proposition}
\begin{proof}
First, note that by Lemma \ref{lem:cohen} the constant $C$ defined in \eqref{equ:tail_bound} is indeed finite.
By refactoring, we have
\[
  \sum_{\vj \geq \vk} \beta(\vj) \vrho^{-\vj} =
  \sum_{\vj \geq \vk} \prod_{m = 1}^M (1 + c j_m)^{\theta} \rho_m^{-j_m}
  =
  \prod_{m=1}^M \left( \sum_{j_m \geq k_m} (1 + c j_m)^{\theta}\rho_m^{-j_m}\right).
\]
We then obtain for each $m=1,\ldots,M$, 
\begin{align*}
	\sum_{j_m \geq k_m} (1 + c j_m)^{\theta}\rho_m^{-j_m}
	& =  (1 + c k_m)^{\theta} \ \rho_m^{-k_m} \ \sum_{j = 0}^{\infty} \left( \frac{1 + c j + c k_m}{1 + c k_m}\right)^{\theta} \rho_m^{-j}\\
	& \leq (1 + c k_m)^{\theta} \ \rho_m^{-k_m} \ \sum_{j = 0}^{\infty} \left(1 + c j\right)^{\theta} \rho_m^{- j}.
\end{align*}
Thus, the refactoring argument can be continued as
\begin{align*}
  	\sum_{\vj \geq \vk} \beta(\vj) \vrho^{-\vj} 
  	& = 
  	\sum_{\vj \geq \vk} \prod_{m = 1}^M (1 + c j_m)^{\theta} \rho_m^{-j_m}\\
	& \leq
	\prod_{m=1}^M \left((1 + c k_m)^{\theta} \ \rho_m^{-k_m} \sum_{j_m \geq 0} (1 + c j_m)^{\theta}\rho_m^{-j_m}\right)\\
	& = \beta(\vk) \vrho^{-\vk} \sum_{\vj \geq 0} \prod_{m = 1}^M (1 + c j_m)^{\theta} \rho_m^{-j_m}
	= C\ \beta(\vk) \vrho^{-\vk},
\end{align*}
with $C$ as in Equation \eqref{equ:tail_bound}.
\end{proof}


\subsection{Proof of Theorem \ref{theo:GN_conv}} \label{sec:convergence proof}

\begin{proof}
We prove Theorem \ref{theo:GN_conv} by applying Theorem \ref{theo:abstract_adapt}.
To this end, we need to verify the three assumptions of Theorem \ref{theo:abstract_adapt}.
The first holds due to Proposition \ref{propo:diane} and the third by construction, cf. Remark \ref{rem:abstract_theorem}.
Hence, it remains to verify the second assumption.
To this end, we set
\begin{equation}\label{equ:eta_hat}
	\widehat \eta_n(\vk) 
	:= 
	\begin{cases}
	\|\Delta_\vk (a \nabla S_{\Lambda_n} u) \|_{L^p_\mu(\vGamma; L^2(D))}, & \vk \in \Lambda_n \cup \Cand_n\\
	0, & \text{ otherwise},
	\end{cases}
\end{equation}
and proceed in two steps (see also Remark \ref{rem:abstract_theorem}):
\begin{enumerate}
\item
We define the (formal) limit
\begin{equation}\label{equ:u_inf}
	u_\infty 
	:=
	\sum_{\vk \in \Lambda_\infty} \Delta_\vk u,
	\qquad
	\Lambda_\infty := \bigcup_{n\in\bbN} \Lambda_n,
\end{equation}
and verify in Lemma \ref{lem:u_inf} below that $u_\infty \in C(\vGamma;H_0^1(D))$ as well as 
\[
	\lim_{n\to \infty}
	\|u_\infty - u_n\|_{C(\vGamma;H_0^1(D))}
	=
	0.
\]

\item
We then set
\begin{equation}\label{equ:eta_inf}
	\eta_\infty(\vk) 
	:= 
	\begin{cases}
	\|\Delta_\vk (a \nabla u_\infty) \|_{L^p_\mu(\vGamma; L^2(D))}, & \vk \in \Lambda_\infty \cup \Marg{\Lambda_\infty},\\
	0, & \text{ otherwise,}
	\end{cases}
\end{equation}
and show in Lemma \ref{lem:eta_inf} that
\[
	\lim_{n\to\infty} \|\eta_\infty - \widehat \eta_n \|_{\ell^1} = 0,
\]
\end{enumerate}
which concludes the proof.
\end{proof}

\begin{lemma}\label{lem:u_inf}
Given the assumptions of Theorem \ref{theo:GN_conv}, \rev{the $u_n$, $n\in\mathbb N$ form a Cauchy sequence in $C(\vGamma;H_0^1(D))$. 
In particular, $u_\infty$ given in \eqref{equ:u_inf} is its well-defined limit in $C(\vGamma;H_0^1(D))$.}
\end{lemma}
\begin{proof}
We abbreviate the norms in $C(\vGamma;H_0^1(D))$ and $C(\vGamma;\bbR)$ by $\|\cdot\|_{C}$.
\rev{Furthermore, let $\vrho \in\bbR^M$ be such that $1 < \rho_m < \alpha^{-1}$ as in equation \eqref{equ:alpha} and let $T_\vk$ and $t_\vk$, $\vk \in \F$, denote the multivariate Taylor polynomials and the corresponding Taylor coefficients of $u$, respectively. For $n,m\in\mathbb{N}$ with $n\leq m$ we obtain by the triangle and Cauchy--Schwarz inequalities
\begin{align*}
	\|u_m - u_n\|_C
	& 
	=
	\left\| S_{\Lambda_m\setminus \Lambda_n} u \right\|_C
	=
	\left\| \sum_{\vk \in \F} t_\vk S_{\Lambda_m\setminus \Lambda_n} T_\vk\right\|_C
	\leq \sum_{\vk \in \F} \|t_\vk\|_{\mc H} \left\| S_{\Lambda_m\setminus \Lambda_n} T_\vk \right\|_{C}\\
	& 
	\leq \left( \sum_{\vk \in \F} \vrho^{2\vk}\ \|t_\vk\|^2_{\mc H} \right)^{1/2} \  \left( \sum_{\vk \in \F} \vrho^{-2\vk} \ \left\| S_{\Lambda_m\setminus \Lambda_n} T_\vk \right\|_{C}^2\right)^{1/2},
\end{align*}
where by Theorem \ref{cor:Bachmayr_finite}
\begin{equation}\label{equ:C_u_rho}
	C_{u,\vrho} := \left( \sum_{\vk \in \F} \vrho^{2\vk}\ \|t_\vk\|^2_{\mc H} \right)^{1/2} < \infty.
\end{equation}
Since $\Delta_\vi T_\vk = 0$ if $i_m > \m^{-1}(k_m)$ for any $m$ we have by Proposition \ref{propo:Delta_vi} and the assumptions that
\begin{align*}
	\left\| S_{\Lambda_m\setminus \Lambda_n} T_\vk \right\|_{C}
	& \leq
	\sum_{\vi \in \Lambda_m\setminus \Lambda_n} \left\| \Delta_\vi T_\vk \right\|_{C}
	\leq
	\sum_{\vi \in \Lambda_\infty\setminus \Lambda_n} \left\| \Delta_\vi T_\vk \right\|_{C}\\
	& = \sum_{\vi \in (\Lambda_\infty\setminus \Lambda_n) \cap \mc R_{\m^{-1}(\vk)}} \left\| \Delta_\vi T_\vk \right\|_{C}\\
	& \leq 
	g_n(\vk) := \sum_{\vi \in (\Lambda_\infty\setminus \Lambda_n) \cap \mc R_{\m^{-1}(\vk)}} \prod_{m=1}^M (1+ck_m)^\theta,
\end{align*}
where $\mc R_{\m^{-1}(\vk)} = \{ \vi \in \F\colon \vi \leq \m^{-1}(\vk)\}$. 
Since for any of the finitely many $\vi \in (\Lambda_\infty\setminus \Lambda_n) \cap \mc R_{\m^{-1}(\vk)}$ there exists an $n_0\in\bbN$ such that $\vi \in \Lambda_n$ for all $n\geq n_0$, we obtain
\[
	\lim_{n\to \infty} g_n(\vk) 
	=
	\lim_{n\to \infty} g^2_n(\vk) 
	=
	0
	\qquad
	\forall \vk \in \F.
\]
Moreover, we conclude as in the proof of Proposition \ref{propo:S_Tk}
\begin{align*}
	g_n(\vk) 
	& 
	\leq \sum_{\vi \in \mc R_{\m^{-1}(\vk)}} \prod_{m=1}^M (1+ck_m)^\theta
	\leq g(\vk) := \prod_{m=1}^M (1+ck_m)^{1+\theta}.
\end{align*}
By Lemma \ref{lem:cohen} we have
\[
	\sum_{\vk \in \F} \vrho^{-2\vk} \  g(\vk)^2 < \infty,
\]
so that $g^2\colon \F \to [0,\infty)$ serves as a summable dominating mapping of the $g^2_n\colon \F \to [0,\infty)$ and we obtain by Lebesgue's dominated convergence theorem
\begin{align*}
	\lim_{n\to\infty}
	\sum_{\vk \in \F} \vrho^{-2\vk} \ g_n(\vk)^2
	=
	0.
\end{align*}
Thus, since
\[
	\|u_m - u_n\|^2_C
	\leq
	C_{u,\vrho}^2
	\sum_{\vk \in \F} \vrho^{-2\vk} \ g_n(\vk)^2
	\qquad \forall m\geq n,
\]
we conclude that the approximations $u_n = \sum_{\vi \in \Lambda_n} \Delta_\vi u$ form a Cauchy sequence in the (complete) Banach space $C(\vGamma;H_0^1(D))$ with $u_\infty = \sum_{\vi \in \Lambda_\infty} \Delta_\vi u$ as its limit, since $\Lambda_n \uparrow \Lambda_\infty$.
}
\end{proof}

For the second step of the proof of Theorem \ref{theo:GN_conv}, we first state an important lemma concerning the decay of the error estimators.

\begin{lemma}\label{lem:eta_bound}
Let the assumptions of Theorem \ref{theo:GN_conv} be satisfied and let $\Lambda\subset \F$ be an arbitrary monotone subset.
Then there exists a constant $C = C(M,\vrho,c,\theta,a)<\infty$ such that for
\[
	\eta(\vk, S_\Lambda u)
	:=
	\|\Delta_\vk (a \nabla S_\Lambda u) \|_{L^p_\mu(\vGamma; L^2(D))},
	\qquad
	\vk\in\F,
\]
we have for any  $\vk \in \F$
\[
	\eta(\vk, S_\Lambda u)
	\leq
	C\ g(\vk), \qquad
	g(\vk)
	:=
	\left( \prod_{m=1}^M (1+c k_m)^{2\theta+1} \right) \ \vrho^{-\vk}.
\]
\end{lemma}
\begin{proof}
Set $u_\Lambda := S_\Lambda u$.
By linearity $\Delta_\vk (a \nabla u_\Lambda)$ for $\vk \in \F$ can be written as
\begin{align*}
	\Delta_\vk \left[a \nabla u_\Lambda\right]
	& = 	\Delta_\vk \left[ a \sum_{\vi \in \Lambda} \Delta_\vi \nabla u\right]
	=  \sum_{\vi \in \Lambda} \Delta_\vk \left[a \Delta_\vi \nabla u\right].
\end{align*}
Moreover, using the Taylor expansion of the solution $u$ we deduce that
\begin{align}
\label{eq:taylor u}
	\Delta_\vk\left[a \Delta_\vi \nabla u\right]
	& = \Delta_\vk\left[a \Delta_\vi \sum_{\vj \in \F} (\nabla t_\vj)\  T_\vj \right]
	= \sum_{\vj \in \F} (\nabla t_\vj)\ \Delta_\vk \left[a\Delta_\vi T_\vj \right].
\end{align}
We observe that for certain combinations of $\vi$, $\vj$, and $\vk$ it holds $\Delta_\vk \left[a \Delta_\vi T_\vj \right] \equiv 0$.
First of all,
\[
	\Delta_\vi T_\vj = \prod_{m = 1}^M (\Delta_{i_m} T_{j_m}) \equiv 0  \qquad \text{ if } \ \exists m\colon j_m \leq \m(i_m-1),
\]
since then $\Delta_{i_{m}} T_{j_{m}} \equiv 0$.
Second, the function $a\Delta_\vi T_\vj$ is a polynomial in $\vy$ belonging to the space 
\[
	\mc P_{\m(\vi) + \bs 1}
	:=
	\mathrm{span}\left\{ \vy^\vp\colon p_m \leq \m(i_m) + 1 \text{ for } m=1,\ldots,M\right\},
\]
since $a$ is affine in $\vy$.
Hence,
\[
	\Delta_\vk \left[a\Delta_\vi T_\vj \right]
	\equiv 0
	\qquad
	\text{if } \ \exists m\colon  \m(i_m) + 1 \leq \m(k_m-1),
\]
We combine now both necessary conditions $\vj \geq \m(\vi-\bs 1) + \bs 1$ and $\m(\vi) + \bs 1 \geq \m(\vk - \bs 1) + \bs 1$ for $\Delta_\vk \left[a \Delta_\vi T_\vj \right] \not\equiv 0$ to 
\[
	\vj \geq \m(\vk - \bs 2)  + \bs 1
	\geq \vk - \bs 1,
\]
where the last inequality follows due to $\m(k)\geq k$.
Thus, introducing the notation $ [\vk -\bs 1]_+ := (\max\{k_m-1, 0\})_{m=1}^M$, the sum~\eqref{eq:taylor u} reduces to
\begin{align*}
	\Delta_\vk \left[a\Delta_\vi u \right]
	= 
	\sum_{\vj \geq [\vk -\bs 1]_+} (\nabla t_\vj)\ \Delta_\vk \left[a\Delta_\vi T_\vj \right].
\end{align*}
\rev{By interchanging the order of summation we obtain}
\begin{align*}
	\left\|\Delta_\vk (a \nabla u_\Lambda)\right\|_{L^p_\mu(\vGamma; L^2(D))}
	& = \left\| \sum_{\vi \in \Lambda} \Delta_\vk ( a \Delta_\vi \nabla u_\Lambda)\right\|_{L^p_\mu(\vGamma; L^2(D))}\\
	& = \left\| \sum_{\vi \in \Lambda} \sum_{\vj \geq [\vk -\bs 1]_+} (\nabla t_\vj)\ \Delta_\vk \left[a\Delta_\vi T_\vj \right] \right\|_{L^p_\mu(\vGamma; L^2(D))}\\
	& = 	\left\| \sum_{\vj \geq [\vk -\bs 1]_+} (\nabla t_\vj)\ \Delta_\vk \left[a S_{\Lambda} T_\vj \right] \right\|_{L^p_\mu(\vGamma; L^2(D))}.
\end{align*}
We now set $\beta(\vk) := \prod_{m = 1}^M (1+c k_m)^{\theta}$  as well as
\begin{equation}\label{equ:a_max}
	a_{\max} := \sup_{y\in\vGamma} \sup_{\vx\in D} |a(\vx,\vy)| < \infty.
\end{equation}
By using the \rev{triangle} inequality, Proposition \ref{propo:Delta_vi} and Proposition \ref{propo:S_Tk} we deduce 
\begin{align*}
	\left\|\Delta_\vk (a \nabla u_\Lambda)\right\|_{L^p_\mu(\vGamma; L^2(D))}
	& = 	\left\| \sum_{\vj \geq [\vk -\bs 1]_+} (\nabla t_\vj)\ \Delta_\vk \left[a S_{\Lambda} T_\vj \right] \right\|_{L^p_\mu(\vGamma; L^2(D))}\\
	& \leq \sum_{\vj \geq [\vk -\bs 1]_+} \|(\nabla t_\vj)\|_{L^2(D)}\ \left\|\Delta_\vk \left[a S_{\Lambda} T_\vj \right]\right\|_{C(\vGamma; \bbR)}\\
	& \leq \sum_{\vj \geq [\vk -\bs 1]_+} \|t_\vj\|_{\mc H}\ \beta(\vk)\ \left\| a S_{\Lambda} T_\vj \right\|_{C(\vGamma; \bbR)}\\
	& \leq \sum_{\vj \geq [\vk -\bs 1]_+} \|t_\vj\|_{\mc H}\ \beta(\vk)\ a_{\max} \left\|S_{\Lambda} T_\vj \right\|_{C(\vGamma; \bbR)}\\
	& \leq a_{\max}\ \beta(\vk) \sum_{\vj \geq [\vk -\bs 1]_+} \|t_\vj\|_{\mc H}\ \gamma(\vj),
\end{align*}
where we set $\gamma(\vj) := \prod_{m = 1}^M (1+cj_m)^{1+\theta}$.
By the Cauchy--Schwarz inequality we obtain
\[
	\sum_{\vj \geq [\vk -\bs 1]_+} \|t_\vj\|_{\mc H}\ \gamma(\vj)
	\leq
	C_{u,\vrho}
	\left(\sum_{\vj \geq [\vk -\bs 1]_+} \vrho^{-2\vj}\ \gamma(\vj)^2 \right)^{1/2},
\]
with $\vrho$ as in Theorem \ref{cor:Bachmayr_finite} and $C_{u,\vrho}$ as in \eqref{equ:C_u_rho}.
We can then apply Proposition \ref{propo:tail_bound} to bound $\sum_{\vj \geq [\vk -\bs 1]_+} \vrho^{-2\vj}\ \gamma(\vj)^2$.
More specifically, Proposition \ref{propo:tail_bound} yields the existence of a constant $C_{\vrho, c, \theta}~<~\infty$ such that
it holds
\[
	\sum_{\vj \geq [\vk -\bs 1]_+} \vrho^{-2\vj}\ \gamma(\vj)^2
	\leq C_{\vrho, c, \theta} \ \vrho^{-2 [\vk -\bs 1]_+ }\ \gamma([\vk -\bs 1]_+ )^2
	\leq C_{\vrho, c, \theta} \left(\prod_{m=1}^M \rho^2_m \right) \vrho^{-2 \vk}\ \gamma(\vk)^2,
\]
since $\gamma$ is increasing and $\rho_m>1$ for each $m$.
Thus, for any $\vk \in \F$ we get
\begin{align*}
	\left\|\Delta_\vk (a \nabla u_\Lambda)\right\|_{L^p_\mu(\vGamma; L^2(D))}
	& \leq
	a_{\max}\ C_{u,\vrho}\ \beta(\vk)\ 
	C_{\vrho, c, \theta}^{1/2} \left(\prod_{m=1}^M \rho_m \right) \gamma(\vk)\ \vrho^{-\vk}.
\end{align*}
The statement follows with
\begin{equation}\label{eq:M-dependent-C}
	C := a_{\max}\ C_{u,\vrho}\ C_{\vrho, c, \theta}^{1/2} \left(\prod_{m=1}^M \rho_m \right),  
\end{equation}
since $g(\vk)= \beta(\vk) \gamma(\vk)\ \vrho^{-\vk}$. 
\end{proof}

This bound of the error indicators is now used to proceed with the second step of the proof to verify the second assumption of Theorem \ref{theo:abstract_adapt}.

\begin{lemma}\label{lem:eta_inf}
Given the assumptions of Theorem \ref{theo:GN_conv} 
we have for $\eta_\infty$ as in \eqref{equ:eta_inf} and $\widehat \eta_n$ as in \eqref{equ:eta_hat} that
\[
	\lim_{n\to \infty}
	\|\eta_\infty - \widehat \eta_n\|_{\ell^1(\F)}
	=
	0.
\]
\end{lemma}
\begin{proof}
We introduce the short-hand notation
\[
	\Lambda^+ := \Lambda \cup \Marg{\Lambda},
	\qquad \Lambda \subseteq \F,
\]
and notice that consequently $\Lambda_\infty^+ \subseteq \bigcup_{n\in\bbN} \Lambda_n^+$. 
Moreover, we have
\[
	\left| \eta_\infty(\vk) - \widehat \eta_n(\vk) \right|
	\leq
	\begin{cases}
	\|\Delta_\vk (a \nabla (u_\infty - u_n))\|_{L^p_\mu(\vGamma; L^2(D))}, & \vk \in \Lambda^+_n \subset \Lambda^+_\infty,\\
	\|\Delta_\vk (a \nabla u_\infty)\|_{L^p_\mu(\vGamma; L^2(D))}, & \vk \in \Lambda^+_\infty \setminus \Lambda^+_n,\\
	0, & \vk \in  \mathcal{F}  \setminus \Lambda^+_\infty.
	\end{cases}
\]
Hence,
\begin{align*}
  \|\eta_\infty - \widehat \eta_n\|_{\ell^1(\F)}
  \leq \,\,
  & \underbrace{\sum_{\vk \in \Lambda_\infty^+} \|\Delta_\vk (a \nabla (u_\infty - u_n))\|_{L^p_\mu(\vGamma; L^2(D))}}_{\mbox{term I}} 
  \,\, + \,\,
  \underbrace{\sum_{\vk \in \Lambda_\infty^+\setminus \Lambda_n^+}  \|\Delta_\vk (a \nabla u_\infty )\|_{L^p_\mu(\vGamma; L^2(D))}}_{\mbox{term II}}.  
\end{align*}
We would like to take the limit on both sides, and verify that the two terms on the right-hand side tend to zero, which we analyze separately in the following.

\paragraph{Term I}
Assuming for a moment that we can apply the dominated convergence theorem to exchange the sum and the limit, we would get
\begin{align*}
  & \lim_{n\to\infty} \sum_{\vk \in \Lambda_\infty^+} \|\Delta_\vk (a \nabla (u_\infty - u_n))\|_{L^p_\mu(\Gamma; L^2(D))} \\
  &  =   \sum_{\vk \in \Lambda_\infty^+} \lim_{n\to\infty} \|\Delta_\vk (a \nabla (u_\infty - u_n))\|_{L^p_\mu(\Gamma; L^2(D))}
  && \mbox{by dominated convergence}\\
  & \leq \sum_{\vk \in \Lambda_\infty^+} \lim_{n\to\infty} \beta(\vk) \|a \nabla (u_\infty - u_n)\|_{C(\vGamma; L^2(D))}
  &&  \mbox{\rev{by Pr. \ref{propo:Delta_vi}}, $\beta(\vk) := \prod_{m = 1}^M (1+c k_m)^{\theta}$} \\
  & \leq \sum_{\vk \in \Lambda_\infty^+} \lim_{n\to\infty}\beta(\vk) a_{\max} \|u_\infty - u_n\|_{C(\vGamma; H_0^1(D))}
  && \mbox{recalling the \rev{def.} of $a_{\max}$ in \eqref{equ:a_max}} \\
  & = 0
  && \mbox{by Lemma~\ref{lem:u_inf}.}
\end{align*}
In order to apply Lebesgue's dominated convergence, we need to check that there exists a function $g\colon \F \to [0,\infty)$ such that, for all $n\in\bbN$ and $\vk\in\Lambda_\infty^+$,
\begin{equation}\label{equ:bound_inf_2}
	\|\Delta_\vk (a \nabla u_\infty )\|_{L^p_\mu(\Gamma; L^2(D))}
	+
	\|\Delta_\vk (a \nabla u_n)\|_{L^p_\mu(\Gamma; L^2(D))}
	\leq 
	g(\vk) 
	\quad
	\text{ and }
	\quad
	\sum_{\vk \in \Lambda_\infty^+}
	g(\vk)
	< \infty.
\end{equation}
The bounding function $g$ is obtained by Lemma \ref{lem:eta_bound}: there exists a constant $C<\infty$ such that
\[
	\|\Delta_\vk (a \nabla u_\infty )\|_{L^p_\mu(\Gamma; L^2(D))}
	+
	\|\Delta_\vk (a \nabla u_n)\|_{L^p_\mu(\Gamma; L^2(D))}
	\leq 
	2C\ g(\vk),
\]
with
\[
	g(\vk)
	:=
	\left( \prod_{m=1}^M (1+c k_m)^{2\theta+1} \right) \ \vrho^{-\vk}.
\] 
The required summability of $g$ is derived by Lemma \ref{lem:cohen}, i.e., 
\[
	\sum_{\vk \in \Lambda_\infty^+}
	2C\ g(\vk)
	\leq 
	2C \sum_{\vk \in \F}
	\left( \prod_{m=1}^M (1+c k_m)^{2\theta+1} \right) \ \vrho^{-\vk}	
	< \infty.
\]

\paragraph{Term II}
To verify that the limit of the second term is also zero, observe that the dominated convergence theorem in \eqref{equ:bound_inf_2} implies
\[
	\sum_{\vk \in \Lambda_\infty^+}
	\|\Delta_\vk (a \nabla u_\infty )\|_{L^p_\mu(\vGamma; L^2(D))}.
	< \infty
\]
Together with the fact that $\Lambda_\infty^+ \subseteq \bigcup_{n\in\bbN} \Lambda_n^+$, 
this implies the final result
\[
	\lim_{n\to \infty}
	\sum_{\vk \in \Lambda_\infty^+\setminus \Lambda_n^+}
	\|\Delta_\vk (a \nabla u_\infty )\|_{L^p_\mu(\Gamma; L^2(D))}
	=
	0. 
\]
\end{proof}

By Lemma~\ref{lem:eta_inf}, the three assumptions of Theorem~\ref{theo:abstract_adapt} have been verified, proving convergence of the described adaptive algorithm.

\subsection{Proof of Theorem \ref{theo:GG_alg}} \label{sec:convergence proof_GG}
\begin{proof}
Again we prove the assertion by applying Theorem \ref{theo:abstract_adapt}, i.e., verifying the three assumptions of Theorem \ref{theo:abstract_adapt}.
The first holds by assumption and the third by construction of Algorithm \ref{alg:GGadap}, cf. Remark \ref{rem:abstract_theorem}.
Thus, it remains again to verify the second assumption of Theorem \ref{theo:abstract_adapt}.
We set
\[
	\Lambda_n^+ := \Lambda_n \cup \Cand_n = \Lambda_n \cup \RMarg{\Lambda_n}
\]
as well as
\begin{equation}\label{equ:eta_hat_GG}
	\widehat \eta_n(\vk) 
	:= 
	\begin{cases}
	\|\Delta_\vk u \|_{L^p_\mu(\vGamma; \mc H)}, & \vk \in \Lambda^+_n\\
	0, & \text{ otherwise},
	\end{cases}
\end{equation}
and define
%
\begin{equation}\label{equ:eta_inf_GG}
	\eta_\infty(\vk) 
	:= 
	\begin{cases}
	\|\Delta_\vk u \|_{L^p_\mu(\vGamma; \mc H)}, & \vk \in \Lambda^+_\infty,\\
	0, & \text{ otherwise,}
	\end{cases},
	\qquad
	\Lambda^+_\infty := \bigcup_{n\in\bbN} \Lambda^+_n.
\end{equation}
We verify in Lemma \ref{lem:GG_aux} below (which is similar to Lemmas \ref{lem:eta_bound} and \ref{lem:eta_inf}) that
\[
	\lim_{n\to\infty} \|\eta_\infty - \widehat \eta_n \|_{\ell^1} = 0,
\]
which concludes the proof.
\end{proof}
\begin{lemma}\label{lem:GG_aux}
Let the assumptions of Theorem \ref{theo:GG_alg} be satisfied. Then, there exists a constant $C< \infty$ such that for any $\vk\in\F$
\begin{equation}\label{equ:surplus_bound}
	\|\Delta_\vk u \|_{L^p_\mu(\vGamma; \mc H)}
	\leq
	C\ g(\vk), \qquad
	g(\vk)
	:=
	\left( \prod_{m=1}^M (1+\widetilde c k_m)^{\widetilde \theta}\,(1+c k_m)^{\theta} \right) \ \vrho^{-\vk}.
\end{equation}
Moreover, we have $(\eta_\infty(\vk))_{\vk\in\F} \in \ell^1(\F)$ for $\eta_\infty(\vk)$ as given in \eqref{equ:eta_inf_GG} and, therefore, for $\widehat \eta_n$ as in \eqref{equ:eta_hat_GG}
\[
	\lim_{n\to\infty} \|\eta_\infty - \widehat \eta_n \|_{\ell^1} = 0.
\]
\end{lemma}
\begin{proof}
In the following we denote the norm in $L^p_\mu(\vGamma; \mc H)$ and $C(\vGamma;\mc H)$ simply by $\|\cdot\|_{L^p}$ and $\|\cdot\|_C$, respectively.
By employing the polynomial expansion of $u$ and the Cauchy--Schwarz inequality, we obtain
\begin{align*}
	\|\Delta_\vk u\|_{L^p}
	& = \left\| \sum_{\vi \in \F} u_\vi \Delta_\vk P_\vi \right\|_{L^p}
	\leq \sum_{\vi \in \F} \|u_\vi\|_{\mc H} \left\|\Delta_\vk P_\vi\right\|_{L^p}\\
	& \leq \left( \sum_{\vi \in \F} \vrho^{2\vi}\ \|u_\vi\|^2_{\mc H} \right)^{1/2} \  \left( \sum_{\vi \in \F}
	\vrho^{-2\vi}\ \left\|\Delta_\vk P_\vi\right\|_{L^p}^2\right)^{1/2},
\end{align*}
where $\vrho \in \bbR^M$ is as assumed in Theorem \ref{theo:GG_alg}.
By assumption the first term is bounded by a constant
\[
	C_{u,\vrho} := \left( \sum_{\vi \in \F} \vrho^{2\vi}\ \|u_\vi\|^2_{\mc H} \right)^{1/2} < \infty.
\]
Concerning the second term, we first note that
\[
	\Delta_\vk P_\vi
	=
	\prod_{m=1}^M \Delta_{k_m} P_{i_m}
	\equiv 0
	\qquad
	\text{ if } \exists m\colon i_m \leq \m(k_m-1).
\]
Hence, we require $\vi \geq \m(\vk - \bs 1) + \bs 1 \geq \vk$ for $\Delta_\vk P_\vi \not\equiv 0$ and therefore obtain by Proposition \ref{propo:Delta_vi} and the assumption \eqref{equ:poly_cond}
\begin{align*}
	\sum_{\vk \in \F} \vrho^{-2\vk}\ \left\|\Delta_\vk P_\vi\right\|_{L^p}^2
	& = \sum_{\vi \geq \vk} \vrho^{-2\vi}\ \left\|\Delta_\vk P_\vi\right\|_{L^p}^2
	\leq \sum_{\vi \geq \vk} \vrho^{-2\vi}\ \left\|\Delta_\vk P_\vi\right\|_{C}^2\\
	& \leq \sum_{\vi \geq \vk} \vrho^{-2\vi}\ \left(\prod_{m=1}^M (1 + c k_m)^\theta\right) \|P_\vi\|_{C(\vGamma;\bbR)}\\
	& \leq \gamma(\vk) \sum_{\vi \geq \vk} \vrho^{-2\vi} \ \beta(\vi)^2,
\end{align*}
with
\[
	\beta(\vi)
	:=
	\prod_{m=1}^M (1 + \widetilde c i_m)^{\widetilde \theta},
	\qquad
	\gamma(\vk)
	:=
	\prod_{m=1}^M (1 + c k_m)^\theta.
\]
Hence, by Proposition \ref{propo:tail_bound} we have for a finite constant $C$
\[
	\sum_{\vi \geq \vk} \vrho^{-2\vi} \ \beta(\vi)^2
	\leq
	C \vrho^{-2\vk} \ \beta(\vk)^2
\]
and, thus,
\[
	\|\Delta_\vk u\|_{L^p}
	\leq
	C_{u,\vrho} \,
	C^{1/2} \
	\gamma(\vk) \,
	 \beta(\vk) \,
	 \vrho^{-\vk}
	 \qquad
	 \vk \in \F,
\]
which proves \eqref{equ:surplus_bound}.
Moreover, by Lemma \ref{lem:cohen} we know that $(g(\vk))_{\vk \in \mc \F} \in \ell^1(\F)$, and hence,
also $(\widehat \eta_n(\vk))_{\vk\in\F}, \ (\eta_\infty(\vk))_{\vk\in\F} \in \ell^1(\F)$, $n\in\bbN$.
Finally, we have by definition of $\eta_\infty$ and $\widehat \eta_n$ that
\begin{align*}
	\|\eta_\infty - \widehat \eta_n\|_{\ell^1}
	& 
	= \sum_{\vk \in \Lambda_\infty^+ \setminus \Lambda_n^+} \|\Delta_\vk u\|_{L^p}
	\leq 	C_{u,\vrho} C^{1/2}\sum_{\vk \in \Lambda_\infty^+ \setminus \Lambda_n^+} g(\vk).
\end{align*}
The summability $(g(\vk))_{\vk \in \mc \F} \in \ell^1(\F)$ and $\Lambda_\infty^+ = \bigcup_{n\in\bbN} \Lambda_n^+$ then yield the desired result
\[
	\lim_{n\to \infty}
	\|\eta_\infty - \widehat \eta_n\|_{\ell^1}
	\leq
	\lim_{n\to \infty}	
	\sum_{\vk \in \Lambda_\infty^+\setminus \Lambda_n^+}
	g(\vk)
	=
	0. 
\]
\end{proof}

\section{Conclusions}\label{sec:conclusions}

We have proved convergence of an adaptive sparse collocation algorithm for approximating the solution of an elliptic PDE with a high-dimensional parameter
\rev{$\vy \in [-1,1]^M$, } applying the analysis technique from \cite{BespalovEtAl2019}, developed for the stochastic Galerkin FEM, to a slight variation of the algorithm
proposed \rev{by Guignard and Nobile} in \cite{GuignardNobile2018}.
In this sense, our work can be seen as an extension of \cite{GuignardNobile2018}, where a very close variant of the algorithm considered here was presented and analyzed numerically, but without convergence proof. 

The algorithm\rev{s} we \rewrite{consider}{propose} here and \rev{that} in \cite{GuignardNobile2018} \rewrite{is}{are both} modifications of the well-known dimension-adaptive sparse grid algorithm of Gerstner and Griebel in that they replace the hierarchical surplus error indicators with a rigorous residual-based error estimator.
As a by-product of our analysis we also obtain a convergence proof for the Gerstner--Griebel algorithm applied to the same problem, under the assumption that the hierarchical surplus error indicator is also a reliable error estimator.
The convergence proof is tailored to the specific problem, i.e., an elliptic PDE with parametric diffusion coefficient depending affinely  on a finite number of parameters.
Because the algorithm is based on a residual-based error estimator, the analysis is problem-specific and
must be adapted for each new PDE as well as for different forms (e.g.\ nonlinear) of the random diffusion coefficient.
However, we expect that a large part of the machinery proves valid or at least extensible in a straightforward way.
Particularly, if reliable error estimators (for the approximation error w.r.t.~the  parameter variables) are available, \rewrite{simply}{only} a stability condition of these estimators w.r.t.~$u_n$ needs to be established in order to verify the crucial second condition of the general convergence Theorem~\ref{theo:abstract_adapt}.
Our analysis in Section~\ref{sec:convergence proof} can serve as a blueprint for doing so.

Regarding possible extensions of this work, we point out that the convergence analysis we have presented proves convergence but does not provide a rate. 
This might be achieved by a saturation assumption following again the line of proof in \cite{BespalovEtAl2019} for adaptive stochastic Galerkin FEM.
Conversely, the extension of the specific model problem to the important case of the diffusion coefficient resulting from the parametrization of a log-normal random field is deemed to be more challenging. 
Another important yet challenging addition to our work would be to extend the convergence result to
the infinite-dimensional case, i.e., to consider countably many parameters $M=\infty$ in the affine expansion of the diffusion coefficient \eqref{equ:affine-a}.
This would pose both theoretical and algorithmic challenges: on the theoretical side, our proof would need to be revisited since some constants are not bounded when $M \rightarrow \infty$ (in particular, the constant $C$ in Lemma \ref{lem:eta_bound}, cf. equation \eqref{eq:M-dependent-C}).
From the algorithmic point of view, having $M=\infty$ would lead to margin sets of infinite cardinality which is, of course, unfeasible. 
Under the assumption that $\| a_m \|_{L^\infty}$ in \eqref{equ:affine-a} are monotone decreasing (this assumption could be weakened),
  then a possible approach would be to implement a so-called ``buffering'' procedure, as discussed in \cite{GuignardNobile2018}
  (see also \cite{SchillingsSchwab2013,ChkifaEtAl2014,NobileEtAl2016,ErnstEtAl2018}): such an algorithm would start considering only the
  first $M_0 < \infty$ parameters, and any time a parameter is ``activated'' (i.e. a collocation point is added along that parameter dimension for the
  first time), the total number of considered parameters would increase by one, in such a way that there are always $M_0$ ``non-activated'' parameters.

A further interesting follow-up would be to carry out an extensive numerical study on a number of different PDEs for which finite element error estimators are available, and investigate numerically whether \rewrite{the algorithm}{Algorithm \ref{alg:asc}} consistently displays good performance (i.e., similar to the GG algorithm) for all the PDEs considered. 
Both these numerical investigations exceed the scope of this work and are left for future research.



\bibliographystyle{amsalpha}
\bibliography{literature}  
\end{document}